\documentclass[11pt]{amsart}

\usepackage{amssymb}
\usepackage{amsthm}
\usepackage{amsmath}
\usepackage{graphicx}
\usepackage{float}
\usepackage{amsaddr}
\usepackage[usenames,dvipsnames]{xcolor}
\usepackage[margin = 1in]{geometry}
\usepackage{enumerate}
\usepackage[toc,page]{appendix}
\usepackage{lineno}
\usepackage{color}
\usepackage{bbm}
\usepackage{bm}
\usepackage{hyperref}
\usepackage{mhchem}
\usepackage{siunitx}
\usepackage{tikz}
\usetikzlibrary{arrows.meta}
\usetikzlibrary{arrows.meta,angles,quotes}
\usepackage{subcaption}

\definecolor{mygreen}{rgb}{0.1,0.75,0.2}

\newtheorem{thm}{Theorem}[section]

\newtheorem{lem}[thm]{Lemma}
\newtheorem{prop}[thm]{Proposition}

\newtheorem{rem}[thm]{Remark}

\numberwithin{equation}{section}

\DeclareMathOperator{\argmin}{argmin}

\newcommand{\pt}{\partial}
\newcommand{\eps}{\varepsilon}
\newcommand{\ud}{\,\mathrm{d}}

\newcommand{\bR}{\mathbb{R}}

\newcommand{\sD}{\mathcal{D}}

\newcommand{\cl}{\scriptscriptstyle{\text{CL}}}
\newcommand{\nlg}{\scriptscriptstyle{\text{LG}}}
\newcommand{\nsg}{\scriptscriptstyle{\text{SG}}}
\newcommand{\nsl}{\scriptscriptstyle{\text{SL}}}

\newcommand{\cH}{\mathcal{H}}

\newcommand{\divS}{\operatorname{div}_{S}}

\begin{document}
\raggedbottom

\title[Phase-field Models and Sharp Interface Limits for Contact Line Dynamics]{Phase-Field Models, Sharp Interface Limits,\\
and Numerical Schemes for Contact Line Dynamics}

    \author{Guosheng Fu}
\address{Department of Applied and Computational Mathematics and Statistics, University of Notre Dame, Notre Dame, IN 46556}
\email{gfu@nd.edu}

\author{Yuan Gao}
\address{Department of Mathematics, Purdue University, West Lafayette, IN 47906}
\email{gao662@purdue.edu}

\author{Jian-Guo Liu}
\address{Department of Mathematics and Department of Physics, Duke University, Durham, NC 27708}
\email{jian-guo.liu@duke.edu}
\date{}

\begin{abstract}

We study phase-field and sharp-interface models for contact line dynamics of a liquid droplet on a solid substrate within a unified variational framework. The motion of the contact line, where liquid, gas, and solid phases meet, presents a fundamental difficulty in continuum modeling due to the classical stress singularity associated with no-slip hydrodynamics. Phase-field models regularize this singularity by introducing a thin transition layer of thickness $\delta$ and encoding interfacial effects through a Ginzburg--Landau free energy augmented by a wall energy on the substrate.

Starting from the total free energy $E = E_b + E_w$, we analyze two prototypical phase-field models: the Allen--Cahn equation and the Cahn--Hilliard equation. Using matched asymptotic expansions as $\delta \to 0$, we   recover their corresponding sharp-interface limits. In the Allen--Cahn case, the limit yields motion by mean curvature with a contact line law driven by deviations of the dynamic contact angle $\theta_{\mathrm{CL}}$ from Young's angle $\theta_Y$. In the Cahn--Hilliard case, the limit leads to a Mullins--Sekerka problem with the same form of contact line dynamics.

A central result of this work is the identification of consistent gradient-flow structures across both models. The Allen--Cahn dynamics correspond to an $L^2$-gradient flow, while the Cahn--Hilliard dynamics correspond to an $H^{-1}$-gradient flow, and both converge to sharp-interface evolutions that preserve the same energy--dissipation structure. This provides a unified interpretation of contact line motion as a consequence of a single  variational principle.

 Finally, we develop energy-stable numerical schemes based on the minimizing movement principle and establish discrete energy dissipation and well-posedness of the fully discrete problem. The numerical examples confirm that the Allen--Cahn and
Cahn--Hilliard schemes both relax toward the same stationary sharp interface  solution while their dynamics reflect the different dissipation mechanism.
\end{abstract}

\keywords{Onsager's principle, Interface and free boundary, Asymptotic analysis, Triple junction}

\maketitle

\tableofcontents
\newpage

\section{Introduction}

The motion of a liquid droplet on a solid substrate is a classical problem in fluid mechanics and materials science, with applications ranging from coating flows and inkjet printing to microfluidics.   A central difficulty lies in the motion of the contact line — the curve where the liquid–gas interface (capillary surface) meets the solid surface. The three surface tensions provide the leading driving force for the geometric motion of droplets; however, the mechanisms and accurate modeling of the competing dynamics between the moving capillary surface and the no-slip or moving contact line condition are not yet fully understood.  Classical hydrodynamic description for droplets with no-slip boundary conditions predict a non-integrable stress singularity at a moving contact line \cite{HuhScriven71}.  To remove this singularity, many different hydrodynamic models with Navier slip boundary condition are proposed to described the contact line moving consistent with microscopic slip mechanism, c.f.,  \cite{QianWangSheng06, RenHuE}.

 As a regularization model without hydrodynamic computations,  phase-field models provide a natural resolution of this difficulty by replacing the sharp liquid--gas interface with a thin transition layer of thickness $\delta > 0$, described by an order parameter $\phi \in [-1,1]$. The two pure phases correspond to $\phi = \pm 1$, while the interface is represented by a smooth transition between these values. Interfacial and wetting effects are encoded through a total free energy of the form
\[
E(\phi) = E_b(\phi) + E_w(\phi),
\]
where $E_b$ is a Ginzburg--Landau bulk energy and $E_w$ is a wall energy defined on the substrate $\Gamma$. This formulation naturally incorporates surface tension coefficients $\sigma_{LG}$, $\sigma_{SL}$, and $\sigma_{SG}$, and provides a variational framework for deriving the governing equations.

Phase-field models for moving contact lines have been extensively studied. Jacqmin \cite{Jacqmin00} derived a thermodynamically consistent formulation of the Cahn--Hilliard/Navier--Stokes system with boundary conditions that incorporate fluid--solid interactions. Qian, Wang, and Sheng \cite{QianWangSheng06} proposed a generalized Navier boundary condition based on molecular dynamics considerations, providing a continuum description consistent with microscopic slip mechanisms. Sharp-interface limits of phase-field models coupled with hydrodynamics have been analyzed, for example, by Yue and Feng \cite{YueFeng10, YueFeng11}, clarifying the role of diffusion and interfacial thickness in regularizing contact line motion.

In the absence of bulk hydrodynamics, Allen--Cahn  and Cahn--Hilliard phase-field models 
for contact line motion were studied by Xu and coauthors \cite{XuEtAl, ChenWangXu}. The sharp interface dynamics were derived via asymptotic expansion for the phase-field models in various settings; see, for instance, \cite{ChenWangXu,LuXu, LuXu21}. Gao and Liu \cite{GaoLiu1,GaoLiu2} provided a rigorous derivation of the contact line dynamics and established the gradient-flow structure of the corresponding sharp interface model. 

  Phase-field models are especially convenient for handling geometric changes of droplets, including merging and splitting where phase transitions happen. Compared with phase-field models, sharp-interface models are more capable of accurately capturing the contact line speed and the mechanism of competing triple surface tensions at the contact line.
The objective of this work is to establish a unified and consistent connection between phase-field models and their sharp-interface limits for contact line dynamics. In particular, we analyze Allen–Cahn and Cahn–Hilliard phase-field models for contact line dynamics and derive their sharp-interface limits, recovering sharp descriptions of contact line mechanism. We employ the Onsager variational principle and gradient-flow structure to unify all the dynamics, and address the relation between phase-field parameters, such as the interface thickness $\delta$ and mobility coefficients, and macroscopic quantities including surface tensions and contact line mobility. This allows one to use the numerically and geometrically tractable phase-field models while accurately recovering sharp descriptions of various physical phenomena in contact line dynamics.

We consider two fundamental phase-field models. The Allen--Cahn model describes non-conservative dynamics driven by an $L^2$-gradient flow of the energy $E(\phi)$, supplemented by a Lagrange multiplier enforcing a volume constraint. The Cahn--Hilliard model, in contrast, describes conservative dynamics corresponding to an $H^{-1}$-gradient flow, in which mass conservation is built into the evolution. In both cases,   the boundary condition on $\Gamma$ arises from the first variation of the wall energy $E_w$ and incorporates the effect of fluid--solid interactions. 

A key aspect of our analysis is the derivation of the sharp-interface limits as $\delta \to 0$ using matched asymptotic expansions. In this limit, the diffuse interface converges to a sharp capillary surface $S_t$ separating the liquid and gas phases, and the bulk free energy converges to the classical surface energy $\sigma_{LG} |S_t|$. The wall energy similarly reduces to $(\sigma_{SL} - \sigma_{SG}) |D_t|$, where $D_t \subset \Gamma$ is the wetted region. The resulting sharp-interface energy is
\[
E(S_t) = \sigma_{LG} |S_t| + (\sigma_{SL} - \sigma_{SG}) |D_t|.
\]
 The first variation expressed in terms of the normal velocity $V$ of capillary  surface and contact line velocity $v_{\mathrm{CL}}$   can be written as
\begin{equation} 
  \frac{\ud E}{\ud t}
  =\sigma_{\nlg}\int_{S_t}HV\,\ud\cH^2
  +\sigma_{\nlg}\int_{\partial D_t}
    (\cos\theta_{\cl}-\cos\theta_Y)\,v_{\cl}\,\ud\cH^1.
\end{equation}

To determine the motion of the droplet with $V$ and $v_{\cl}$, we identify a consistent variational structure across both phase-field models and their sharp-interface limits.

At the phase-field level, both the Allen--Cahn and Cahn--Hilliard equations are derived from the Onsager variational principle, which balances the rate of energy decrease with a dissipation functional. 
In the sharp-interface limit, this structure is preserved: the Allen--Cahn model corresponds to an $L^2$-gradient flow of $E(S_t)$, while the Cahn--Hilliard model corresponds to an $H^{-1}$-type gradient flow associated with the Mullins--Sekerka dynamics.
For the Allen--Cahn model, the sharp-interface limit yields motion by mean curvature in the bulk, with normal velocity $V$ satisfying
\[
\xi V = -\sigma_{LG} H - \Lambda,
\]
together with a contact line law relating the contact line velocity $v_{\mathrm{CL}}$ to the deviation of the dynamic contact angle $\theta_{\mathrm{CL}}$ from Young's angle $\theta_Y$
$$\zeta v_{\cl}
=
\sigma_{\nlg}(\cos\theta_Y - \cos\theta_{\cl}).$$
For the Cahn--Hilliard model, the limit leads to a Mullins--Sekerka problem, in which the chemical potential is harmonic in each phase and the interface velocity is determined by the jump of its normal derivative, while the same form of contact line law is retained on $\partial D_t$.

    Our derivation demonstrates that the phase-field and sharp-interface descriptions are fully consistent at the level of energetic variational formulation.

An important geometric feature of the problem is the relation between the normal velocity $V$ of the interface and the contact line velocity $v_{\mathrm{CL}}$ along the substrate. These are related by the kinematic identity
\[
V = v_{\mathrm{CL}} \sin \theta_{\mathrm{CL}},
\]
which reflects the attachment of the interface to the substrate.   This relation plays a crucial role in connecting the phase-field formulation with the sharp-interface description. This is actually a necessary condition to derive the sharp-interface with correct contact line motion  $v_{\mathrm{CL}}$. In particular, it explains the appearance of geometric factors such as $\sin \theta_{\mathrm{CL}}$ in intermediate formulations and ensures that all derived contact line laws are physically equivalent.   The attachment of the droplets on the substrate is assumed as a kinematic condition in the sharp interface description.   However, the conditions in the  phase-field description to ensure the droplet   attached on the substrate during dynamics  is still unknown.

Finally, we develop numerical schemes that preserve the variational structure of the models. Using a minimizing movement formulation, we construct fully discrete schemes based on $H^1$-conforming finite elements and prove discrete energy dissipation. Under a constraint for time step, the variational problem becomes convex, and thus it is implemented by the corresponding Euler-Lagrange equation.  We also introduce a convex-splitting scheme that is unconditionally stable. These results provide a robust computational framework consistent with the underlying analysis. The numerical examples confirm that the Allen--Cahn and
Cahn--Hilliard schemes relax toward the same stationary sharp interface spherical cap solution, while their dynamic paths to equilibrium reflect the different bulk dissipation.

The remainder of the paper is organized as follows. Section~2 introduces the phase-field and sharp-interface free energies and their first variations. Sections~3 and~4 analyze the Allen--Cahn and Cahn--Hilliard models, respectively, together with their sharp-interface limits. Section~5 presents the numerical schemes and their stability properties.

 \section{Free Energies: Phase-Field and Sharp Interface Models}
\label{sec:energy}

In this section, we first introduce the total free energies governing contact
line dynamics in both the phase-field and sharp-interface settings. The total
energy in the phase-field model consists of a bulk energy with a
Ginzburg--Landau potential and a wall energy encoding the preference between
the bistable states. For the small parameter $\delta$, we review the
$\Gamma$-convergence result for the phase-field energy, which yields the
limiting sharp-interface energy representing the surface energy of the
droplet. Finally, we compute the first variations of both energies. In
particular, the first variation of the sharp-interface energy is expressed as
a dual pairing of unbalanced force and velocity on both the capillary surface and
the contact line; see Theorem \ref{thm:first_variation_varifold}.

\subsection{Phase-field free energy}
\label{ssec:pf_energy}

Let $\Omega \subset \bR^3$ be a bounded Lipschitz domain representing the
container occupied by a two-phase fluid. The solid substrate is identified with
the flat portion
\[
  \Gamma := \{(x,y,z)\in\partial\Omega : z = 0\}.
\]
The phase-field variable $\phi : \Omega \times [0,\infty) \to [-1,1]$ encodes
the local composition of the fluid: $\phi = 1$ corresponds to the gas phase,
and $\phi = -1$ corresponds to the liquid phase. Across the liquid--gas interface,
$\phi$ varies smoothly over a thin transition layer of width $\delta > 0$.

\subsubsection*{Bulk energy}

The bulk free energy is defined in terms of the double-well Ginzburg--Landau
potential
\begin{equation}\label{eq:GL}
  F(\phi) = \tfrac{1}{4}(\phi^2 - 1)^2,
\end{equation}
which has global minima at $\phi = \pm 1$, corresponding to the two pure phases.

Throughout this paper, $\sigma_{\nlg}$, $\sigma_{\nsg}$, and $\sigma_{\nsl}$
denote the liquid--gas, solid--gas, and solid--liquid surface tension
coefficients, respectively, and are assumed to be positive constants. The
interface thickness parameter $\delta > 0$ will later be sent to zero in the
sharp-interface limit.

The bulk energy functional is given by
\begin{equation}\label{eq:E_b}
  E_b(\phi)
  := \frac{3\sqrt{2}}{4}\,
     \sigma_{\nlg}\int_\Omega \left[\frac{\delta}{2}|\nabla\phi|^2 + \frac{1}{\delta}F(\phi)\right]
     \ud^3 x,
\end{equation}
where the far-field boundary condition on $\partial\Omega\setminus\Gamma$ is
taken to be either $\partial_n\phi=0$ or $\phi=1$.

The prefactor $\frac{3\sqrt{2}}{4}$ is the normalization constant arising from
the classical $\Gamma$-convergence result of Modica~\cite{Modica87}. As
$\delta \to 0$,
\begin{equation}\label{eq:Gamma_lim}
  \Gamma\text{-}\lim_{\delta\to 0}\;
  \frac{3\sqrt{2}}{4}\, \sigma_{\nlg}
\int_\Omega\left[\frac{\delta}{2}|\nabla\phi|^2 + \frac{1}{\delta}F(\phi)\right]
  = \sigma_{\nlg}|S_t|,
\end{equation}
where $|S_t|$ denotes the two-dimensional Hausdorff measure of the sharp
liquid--gas interface $S_t$. Thus, in the sharp-interface limit, the diffuse
bulk energy converges to the classical surface energy of the capillary
interface.

\subsubsection*{Wall energy on the substrate}

To model fluid--solid interactions on the substrate $\Gamma$, we adopt the wall
energy density proposed by Jacqmin~\cite{Jacqmin00}:
\begin{equation}\label{eq:GLw}
  f_w(\phi) = \frac{\phi^3-3\phi}{4}+ \frac{1}{2}.
\end{equation}
This function satisfies
\[
  f_w(1) = 0,
  \qquad
  f_w(-1) = 1,
\]
so that the gas and liquid phases correspond to the two limiting values of the
wall energy density. Its derivative is
\begin{equation}\label{eq:f_w'}
  f_w'(\phi) = \frac{3}{4}(\phi^2-1),
\end{equation}
which vanishes at $\phi=\pm1$.

The associated wall energy functional is
\begin{equation}\label{eq:E_w}
  E_w(\phi)
  := (\sigma_{\nsl} - \sigma_{\nsg})\int_\Gamma f_w(\phi)\,\ud^2 x.
\end{equation}

In the sharp-interface limit $\delta \to 0$, the phase-field approaches
$\phi \approx 1$ in the gas region and $\phi \approx -1$ in the liquid region.
Assume that the sharp interface $S_t$ meets the substrate $\Gamma$, and let
$D_t \subset \Gamma$ denote the wetting region occupied by liquid on the
substrate. Its boundary $\partial D_t$ is the contact line. Since
$f_w(1)=0$ and $f_w(-1)=1$, the wall energy formally approximates the substrate
surface energy in the sharp-interface model:
\begin{equation}\label{Eph2}
  E_w(\phi) \approx (\sigma_{\nsl} - \sigma_{\nsg})|D_t|,
\end{equation}
where $|D_t|$ denotes the two-dimensional Lebesgue measure of $D_t$ on the
substrate.

\subsubsection*{Total phase-field free energy and its first variation}

The total free energy of the phase-field system is the sum of the bulk and wall
contributions:
\begin{equation}\label{eq:E_total}
  E = E_b + E_w.
\end{equation}
To derive the governing equations, we compute its first variation. For any test
function $\tilde\phi \in H^1(\Omega)$, integration by parts yields
\begin{align}
  \frac{\mathrm{d}}{\mathrm{d}\eps}\bigg|_{\eps=0} E(\phi + \eps\tilde\phi)
  &= \frac{3\sqrt{2}}{4}\,\sigma_{\nlg}
     \int_\Omega\left[\frac{1}{\delta}F'(\phi) - \delta\Delta\phi\right]
     \tilde\phi\,\ud^3 x  \notag\\
  &\quad
  + \frac{3\sqrt{2}}{4}\,\sigma_{\nlg}\,\delta
    \int_\Gamma \partial_n\phi\,\tilde\phi\,\ud^2 x  \notag\\
  &\quad
  + (\sigma_{\nsl} - \sigma_{\nsg})\int_\Gamma f_w'(\phi)\tilde\phi\,\ud^2 x .
  \label{eq:first_var}
\end{align}
Here $\partial_n\phi=-\partial_z\phi$ denotes the outward normal derivative on $\Gamma=\{z=0\}$, with
the normal pointing into the solid substrate.

\subsection{Sharp interface free energy and its first variation}
\label{ssec:si_energy_varifold}

We clarify the notation and conventions used throughout.

The interface $S_t$ is assumed to be a $C^2$ embedded surface with boundary $\partial S_t=\partial D_t$. The unit vector $n$ denotes the outward normal to $S_t$, pointing from liquid to gas. The quantity $V:=X\cdot n$ is the scalar normal velocity of $S_t$ induced by a velocity field $X$. The scalar mean curvature $H$ is defined following the convention that a sphere of radius $R$ with outward unit normal has $H=2/R>0$.

At a contact point $p$ on $\partial S_t=\partial D_t$, let $\tau$ be the positively oriented unit tangent vector, let $n_{\cl}$ be the outward unit normal to $\partial D_t$ within the substrate $\Gamma$, and let $b$ be the outward unit co-normal to $\partial S_t$ within the tangent plane $T_pS_t$.
The contact angle $\theta_{\cl}\in(0,\pi)$ is defined by
\begin{equation}\label{eq:contact_angle_def}
  b\cdot n_{\cl}=\cos\theta_{\cl}.
\end{equation}
Young's angle $\theta_Y$ is determined by Young's law
\begin{equation}\label{eq:Young_law}
  \cos\theta_Y=\frac{\sigma_{\nsg}-\sigma_{\nsl}}{\sigma_{\nlg}}.
\end{equation}
If the contact line moves on the substrate with normal speed $v_{\cl}$ in
the direction $n_{\cl}$, then the attachment condition implies
\begin{equation}\label{eq:contact_kinematic}
  V|_{\partial S_t}=v_{\cl}\sin\theta_{\cl}.
\end{equation}

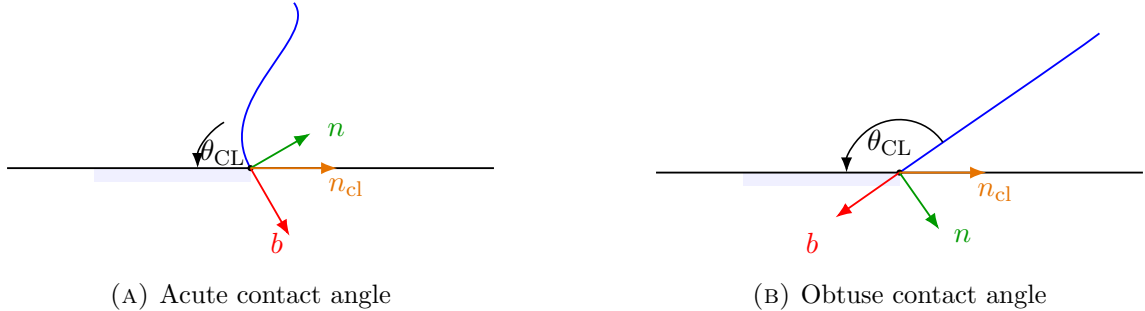
\begin{figure}[htbp]
\centering

\begin{subfigure}[t]{0.48\textwidth}
\centering
\begin{tikzpicture}[scale=1.15,>=Latex]

\tikzset{
  line/.style={line width=0.7pt},
  vec/.style={->, line width=0.7pt},
}

\coordinate (O) at (0,0);
\coordinate (L) at (-2.8,0);
\coordinate (R) at (2.8,0);

\def\th{120}

\coordinate (T1) at ({1.3*cos(\th)},{1.3*sin(\th)});
\coordinate (C1) at ({0.9*cos(\th)},{0.9*sin(\th)});
\coordinate (C2) at (0.8,1.5);
\coordinate (P2) at (0.5,1.9);

\coordinate (Bdir) at ({-0.9*cos(\th)},{-0.9*sin(\th)});
\coordinate (Ndir) at ({0.8*sin(\th)},{-0.8*cos(\th)});
\coordinate (Ncldir) at (1.0,0);
\coordinate (Lin) at (-1.1,0);

\fill[blue!6] (-1.8,-0.15) rectangle (0,0);
\draw[line] (L) -- (R);
\draw[line,blue] (O) .. controls (C1) and (C2) .. (P2);

\fill (O) circle (1pt);

\draw[vec,red]
  (O) -- (Bdir)
  node[pos=1.10,left] {$b$};

\draw[vec,green!60!black]
  (O) -- (Ndir)
  node[pos=1.10,right] {$n$};

\draw[vec,orange!90!black]
  (O) -- (Ncldir)
  node[pos=1.10,below] {$n_{\mathrm{cl}}$};

\pic[
  draw,
  ->,
  line width=0.6pt,
  angle radius=0.7cm,
  "$\theta_{\mathrm{CL}}$"
] {angle = T1--O--Lin};

\end{tikzpicture}
\caption{Acute contact angle}
\end{subfigure}
\hfill
\begin{subfigure}[t]{0.48\textwidth}
\centering
\begin{tikzpicture}[scale=1.15,>=Latex]

\tikzset{
  line/.style={line width=0.7pt},
  vec/.style={->, line width=0.7pt},
}

\coordinate (O) at (0,0);
\coordinate (L) at (-2.8,0);
\coordinate (R) at (2.8,0);

\def\th{35}

\coordinate (T1) at ({1.3*cos(\th)},{1.3*sin(\th)});
\coordinate (C1) at ({0.9*cos(\th)},{0.9*sin(\th)});
\coordinate (C2) at (1.9,1.3);
\coordinate (P2) at (2.3,1.6);

\coordinate (Bdir) at ({-0.9*cos(\th)},{-0.9*sin(\th)});
\coordinate (Ndir) at ({0.8*sin(\th)},{-0.8*cos(\th)});
\coordinate (Ncldir) at (1.0,0);
\coordinate (Lin) at (-1.1,0);

\fill[blue!6] (-1.8,-0.15) rectangle (0,0);
\draw[line] (L) -- (R);
\draw[line,blue] (O) .. controls (C1) and (C2) .. (P2);

\fill (O) circle (1pt);

\draw[vec,red]
  (O) -- (Bdir)
  node[pos=1.10,below left] {$b$};

\draw[vec,green!60!black]
  (O) -- (Ndir)
  node[pos=1.10,right] {$n$};

\draw[vec,orange!90!black]
  (O) -- (Ncldir)
  node[pos=1.10,below] {$n_{\mathrm{cl}}$};

\pic[
  draw,
  ->,
  line width=0.6pt,
  angle radius=0.7cm,
  "$\theta_{\mathrm{CL}}$"
] {angle = T1--O--Lin};

\end{tikzpicture}
\caption{Obtuse contact angle}
\end{subfigure}

\caption{Contact line geometry for acute and obtuse contact angles $\theta_{\mathrm{CL}}$, measured inside the liquid region between the substrate $\Gamma$ and the tangent to the interface $S_t$. The orange vector $n_{\mathrm{cl}}$ is the outward unit normal to $\partial D_t$ within the substrate $\Gamma$. The red vector $b$ is the outward unit co-normal to $\partial S_t$ within the tangent plane $T_pS_t$. The green vector $n$ is the outward unit normal to $S_t$, pointing from the liquid to the gas phase.}
\label{fig:contact_angle_both}
\end{figure}

Combining the liquid--gas surface energy with the solid wetting contribution, the
sharp-interface total energy is
\begin{equation}\label{eq:E_sharp_varifold}
  E(S_t)=\sigma_{\nlg}|S_t|+(\sigma_{\nsl}-\sigma_{\nsg})|D_t|.
\end{equation}

We now give a derivation of the first variation formula using the language
of rectifiable varifolds. This makes precise the decomposition into bulk mean-curvature
and contact-line contributions.

\subsubsection*{Varifold setup}

Let $\mathcal{V}$ be a $2$-dimensional rectifiable varifold in $\bR^3$ with weight measure
$\|\mathcal{V}\|$. For any vector field $X\in C_c^1(\bR^3;\bR^3)$, the first variation of $\mathcal V$
is defined by
\begin{equation}\label{eq:varifold_first_variation_general}
  \delta \mathcal V(X)
  :=
  \int_{\bR^3\times G(2,3)}
  \operatorname{div}_S X(x)\,\ud \mathcal V(x,S),
\end{equation}
where $\operatorname{div}_S X(x)$ denotes the tangential divergence of $X$
with respect to the plane $S\in G(2,3)$.

Equivalently, when $\mathcal V$ is rectifiable, one has the representation
\begin{equation}\label{eq:varifold_first_variation_weight}
  \delta \mathcal V(X)
  =
  \int \operatorname{div}_S X \,\ud\|\mathcal V\|.
\end{equation}
We refer to \cite[Sec.~3]{Allard72}, \cite[Sec.~8.1]{Simon83}, and
\cite[Ch.~17]{Maggi12} for these definitions and properties.

\medskip

In the present smooth setting, let
\[
\mathcal V_t=\mathbf{v}(S_t,1)
\]
denote the multiplicity-one rectifiable varifold associated with a $C^2$ surface
$S_t$, i.e.\ the density function satisfies $\theta\equiv 1$. Then
\begin{equation}\label{eq:varifold_first_variation_smooth}
  \|\mathcal V_t\|=\cH^2\llcorner S_t,
  \qquad
  \delta \mathcal V_t(X)
  =
  \int_{S_t}\operatorname{div}_{S_t}X\,\ud\cH^2.
\end{equation}
This formula is standard; however, for completeness, we give the following two
lemmas to make the computations more explicit and transparent. The first lemma
gives an intuitive computation in the smooth setting using a perturbed flow
map. The second lemma explicitly computes the interior and boundary
contributions in the first variation.

\medskip

To make \eqref{eq:varifold_first_variation_smooth} more intuitive, it is useful
to interpret $X$ as generating a deformation of space. Let
$\{\Phi_\eps\}_{|\eps|<\eps_0}$ be the flow of diffeomorphisms associated with $X$,
defined by
\begin{equation}\label{eq:flow_map}
  \frac{\ud}{\ud\eps}\Phi_\eps(x)=X(\Phi_\eps(x)),
  \qquad
  \Phi_0(x)=x.
\end{equation}
The map $\Phi_\eps$ transports both the surface and its tangent planes, and the
first variation $\delta \mathcal V(X)$ measures the first-order change of the mass of the
varifold under this deformation. In this sense, $\delta \mathcal V(X)$ can be interpreted as
the directional derivative of area in the direction of $X$.

\medskip

\begin{lem}\label{lem2.1}
In this smooth setting, we have the identity
\begin{equation}\label{eq:flow_interpretation_smooth_appendix}
  \delta \mathcal V_t(X)
  =
  \left.\frac{\ud}{\ud\eps}\right|_{\eps=0}|S_\eps|,
\end{equation}
where $S_\eps := \Phi_\eps(S_t)$.
\end{lem}
\begin{proof}

By the area formula,
\begin{equation}\label{eq:area_formula}
  |S_\eps|
  =
  \int_{S_t} J^{S_t}\Phi_\eps(x)\,\ud\cH^2(x),
\end{equation}
where $J^{S_t}\Phi_\eps(x)$ is the tangential Jacobian of $\Phi_\eps$ along $S_t$,
given by
\begin{equation}\label{Jacob}
  J^{S_t}\Phi_\eps(x)
  = \sqrt{\det\bigl(\langle D\Phi_\eps(x)\,\tau_i, D\Phi_\eps(x)\tau_j)\rangle\bigr)},
\end{equation}
for any orthonormal basis $\{\tau_1,\tau_2\}$ of $T_x S_t$.

Since $\Phi_\eps(x)=x+\eps X(x)+o(\eps)$, we have
\[
  D\Phi_\eps(x)=I+\eps\,DX(x)+o(\eps),
\]
and thus
\[
  D\Phi_\eps(x)\tau_i
  =
  \tau_i+\eps\,DX(x)\tau_i+o(\eps).
\]
Consequently,
\[
  \langle D\Phi_\eps\tau_i, D\Phi_\eps\tau_j\rangle
  =
  \delta_{ij}
  +
  \eps\bigl(
    \langle DX\tau_i,\tau_j\rangle
    +
    \langle \tau_i,DX\tau_j\rangle
  \bigr)
  + o(\eps).
\]
Using $\det(I+\eps A)=1+\eps\,\mathrm{tr}(A)+o(\eps)$, \eqref{Jacob} becomes
\[
  J^{S_t}\Phi_\eps(x)
  =
  1 + \eps\sum_{i=1}^2 \langle DX(x)\tau_i,\tau_i\rangle + o(\eps).
\]
Hence,
\begin{equation}\label{eq:jacobian_expansion}
  \left.\frac{\ud}{\ud\eps}\right|_{\eps=0} J^{S_t}\Phi_\eps(x)
  =
  \sum_{i=1}^2 \langle DX(x)\tau_i,\tau_i\rangle
  =:
  \operatorname{div}_{S_t}X(x).
\end{equation}

Differentiating \eqref{eq:area_formula} and using
\eqref{eq:jacobian_expansion}, we obtain
\[
  \left.\frac{\ud}{\ud\eps}\right|_{\eps=0}|S_\eps|
  =
  \int_{S_t}\operatorname{div}_{S_t}X\,\ud\cH^2.
\]
Therefore,
\[
  \delta \mathcal V_t(X)
  =
  \left.\frac{\ud}{\ud\eps}\right|_{\eps=0}|S_\eps|.
\]
\end{proof}
\medskip

This shows that, in the smooth setting, the abstract varifold first variation
coincides with the classical first variation of surface area under the flow
generated by $X$.

\medskip

Because $S_t$ has a boundary, $\delta \mathcal V_t$ contains both an interior curvature term
and a boundary term. We first give the precise formula.
\begin{lem}\label{lemV}
The first variation of the varifold $\mathcal{V}$ is
\begin{equation}\label{eq:area_derivative_varifold}
  \frac{\ud}{\ud t}|S_t|
  =\delta \mathcal V_t(X) =\int_{S_t}HV\,\ud\cH^2+\int_{\partial S_t}b\cdot X\,\ud\cH^1,
\end{equation}
where $b$ is the outward unit co-normal to
  $\partial S_t$ within $TS_t$.
\end{lem}
Below, we derive this formula using the decomposition of the vector field $X$.
Thus, the rate of change of area decomposes into a bulk term governed by mean
curvature and a boundary term governed by the motion of the contact line.

\begin{proof}[Proof of Lemma \ref{lemV}]

For clarity, let us derive \eqref{eq:area_derivative_varifold} by decomposing
$X$ into normal and tangential parts. Write
\begin{equation}\label{eq:X_decomposition}
  X=(X\cdot n)n+X^\top,
  \qquad X^\top\in TS_t.
\end{equation}
Then
\[
  \divS X=\divS\big((X\cdot n)n\big)+\divS X^\top.
\]
Hence
\begin{equation}\label{eq:split_delta}
  \delta \mathcal V_t(X)
  =\int_{S_t}\divS\big((X\cdot n)n\big)\,\ud\cH^2
   +\int_{S_t}\divS X^\top\,\ud\cH^2.
\end{equation}

\paragraph{Normal part.}
Let $f:=X\cdot n$. Since $n$ is normal to $S_t$, one has the standard identity
\[
  \divS(fn)=f\,\divS n.
\]
Under our convention $\divS n=H$, hence
\begin{equation}\label{eq:normal_part}
  \int_{S_t}\divS\big((X\cdot n)n\big)\,\ud\cH^2
  =\int_{S_t}H(X\cdot n)\,\ud\cH^2
  =\int_{S_t}HV\,\ud\cH^2.
\end{equation}

\paragraph{Tangential part.}
Since $X^\top$ is tangent to $S_t$, the surface divergence theorem gives
\begin{equation}\label{eq:tangent_part}
  \int_{S_t}\divS X^\top\,\ud\cH^2
  =\int_{\partial S_t} X^\top\cdot b\,\ud\cH^1
  =\int_{\partial S_t} X\cdot b\,\ud\cH^1,
\end{equation}
because $b\in TS_t$.

Combining \eqref{eq:split_delta}, \eqref{eq:normal_part}, and
\eqref{eq:tangent_part}, we obtain \eqref{eq:area_derivative_varifold}.
\end{proof}

\subsubsection*{First variation of the total energy}

We now combine the above lemmas to compute the variation of total energy
\eqref{eq:E_sharp_varifold}.

\begin{thm}[First variation of total energy]\label{thm:first_variation_varifold}
Let
\[
  E(S_t)=\sigma_{\nlg}|S_t|+(\sigma_{\nsl}-\sigma_{\nsg})|D_t|.
\]
Then
\begin{equation}\label{eq:dEdt_varifold_1}
  \frac{\ud E}{\ud t}
  =\sigma_{\nlg}\int_{S_t}HV\,\ud\cH^2
  +\sigma_{\nlg}\int_{\partial D_t}
    (\cos\theta_{\cl}-\cos\theta_Y)\,v_{\cl}\,\ud\cH^1.
\end{equation}
\end{thm}

\begin{proof}
First, differentiate \eqref{eq:E_sharp_varifold}:
\begin{equation}\label{eq:energy_diff_start}
  \frac{\ud E}{\ud t}
  =\sigma_{\nlg}\frac{\ud}{\ud t}|S_t|
   +(\sigma_{\nsl}-\sigma_{\nsg})\frac{\ud}{\ud t}|D_t|.
\end{equation}

Second, we express the boundary term in terms of the
contact line speed.

Using Lemma \ref{lemV},
we now compute $b\cdot X$ on $\partial S_t$ in terms of the contact line velocity.

Because the substrate is fixed, admissible variations must satisfy
\begin{equation}\label{eq:substrate_tangent_condition}
  X(x)\in T_x\Gamma
  \qquad\text{for }x\in\partial S_t.
\end{equation}
Therefore, the contact line moves within $\Gamma$, and its instantaneous velocity is
\begin{equation}\label{eq:contact_velocity_vector}
  X|_{\partial S_t}=v_{\cl}\,n_{\cl} + w\,\tau
\end{equation}
for some tangential reparametrization speed $w$ along the contact line. Since
$b\perp \tau$, only the $n_{\cl}$-component contributes:
\begin{equation}\label{eq:b_dot_X_1}
  b\cdot X = v_{\cl}\, b\cdot n_{\cl}.
\end{equation}
By the definition \eqref{eq:contact_angle_def} of $\theta_{\cl}$,
\begin{equation}\label{eq:b_dot_X_2}
  b\cdot X = v_{\cl}\cos\theta_{\cl}.
\end{equation}
Substituting into \eqref{eq:area_derivative_varifold} yields
\begin{equation}\label{eq:area_derivative_full}
  \frac{\ud}{\ud t}|S_t|
  =\int_{S_t}HV\,\ud\cH^2
   +\int_{\partial S_t}\cos\theta_{\cl}\,v_{\cl}\,\ud\cH^1.
\end{equation}

Third, we compute the variation of the wetting area $|D_t|$
by differentiating the substrate term $|D_t|$.

Since $D_t\subset \Gamma$ is a
moving domain on the fixed surface $\Gamma$, transported by the boundary speed
$v_{\cl}n_{\cl}$, the transport theorem on the surface $\Gamma$ gives
\begin{equation}\label{eq:wetting_area_derivative}
  \frac{\ud}{\ud t}|D_t|
  =\int_{\partial D_t} v_{\cl}\,\ud\cH^1.
\end{equation}
This is the $2$-dimensional Reynolds transport formula on the substrate.

Finally, combining \eqref{eq:area_derivative_full} and \eqref{eq:wetting_area_derivative},
\begin{align}
  \frac{\ud E}{\ud t}
  &=\sigma_{\nlg}\left(
      \int_{S_t}HV\,\ud\cH^2
      +\int_{\partial D_t}\cos\theta_{\cl}\,v_{\cl}\,\ud\cH^1
    \right)
   +(\sigma_{\nsl}-\sigma_{\nsg})
    \int_{\partial D_t}v_{\cl}\,\ud\cH^1 \notag\\
  &=\sigma_{\nlg}\int_{S_t}HV\,\ud\cH^2
   +\int_{\partial D_t}
      \Big(
        \sigma_{\nlg}\cos\theta_{\cl}
        +\sigma_{\nsl}-\sigma_{\nsg}
      \Big)
      v_{\cl}\,\ud\cH^1.
  \label{eq:energy_diff_mid}
\end{align}
By Young's law \eqref{eq:Young_law},
\[
  \sigma_{\nsl}-\sigma_{\nsg}
  =-\sigma_{\nlg}\cos\theta_Y.
\]
Hence
\[
  \sigma_{\nlg}\cos\theta_{\cl}+\sigma_{\nsl}-\sigma_{\nsg}
  =\sigma_{\nlg}(\cos\theta_{\cl}-\cos\theta_Y).
\]
Therefore
\[
  \frac{\ud E}{\ud t}
  =\sigma_{\nlg}\int_{S_t}HV\,\ud\cH^2
  +\sigma_{\nlg}\int_{\partial D_t}
   (\cos\theta_{\cl}-\cos\theta_Y)v_{\cl}\,\ud\cH^1.
\]
This is \eqref{eq:dEdt_varifold_1}. \end{proof}

\begin{rem}
The formula
\[
  \frac{\ud E}{\ud t}
  =\sigma_{\nlg}\int_{S_t}HV\,\ud\cH^2
  +\sigma_{\nlg}\int_{\partial D_t}
   (\cos\theta_{\cl}-\cos\theta_Y)v_{\cl}\,\ud\cH^1
\]
contains two contributions: the bulk mean-curvature term
\[
  \sigma_{\nlg}\int_{S_t}HV\,\ud\cH^2,
\]
which is the first variation of area in the interior, and the contact-line term
\[
  \sigma_{\nlg}\int_{\partial D_t}
  (\cos\theta_{\cl}-\cos\theta_Y)v_{\cl}\,\ud\cH^1,
\]
which represents the uncompensated Young force and vanishes when
$\theta_{\cl}=\theta_Y$. Thus the first variation splits into an interior
mean-curvature contribution and a boundary wetting contribution.
\end{rem}
 \section{Allen--Cahn Model: Phase-Field and Sharp-Interface}
\label{sec:AC}

In this section, we introduce a conservative phase-field Allen--Cahn model by
specifying a quadratic dissipation functional together with a global volume
constraint. We then derive the corresponding sharp-interface limit through
formal asymptotic analysis. The limiting sharp-interface model is governed by
mean curvature flow coupled with contact line dynamics, which recovers the
accurate contact line speed mechanism. We also give the variational structure
of the sharp-interface model.

\subsection{Phase-field Allen--Cahn model}

To describe dissipative dynamics, we introduce the dissipation functional
\begin{equation}\label{eq:diss_AC}
  {\mathcal D}_{\text{AC}}(u, u|_\Gamma)
  = \frac{3\sqrt{2}}{4}\,\frac{\xi \delta}2
    \int_\Omega u^2\,\ud^3 x
  + \frac{3\sqrt{2}}{4}\,\frac{\zeta\delta}2
    \int_\Gamma (u|_\Gamma)^2\,\ud^2 x,
\end{equation}
where $u = \partial_t\phi$ and $u|_\Gamma = \partial_t\phi|_\Gamma$ are the
bulk and surface rates of change, respectively. We keep the prefactor
$\frac{3\sqrt{2}}{4}$ as in the definition of the free energy
\eqref{eq:E_b} for both phase-field models. The friction coefficients
$\xi, \zeta > 0$ are associated with the bulk and substrate dissipation.
The factors of $\delta$ ensure that the dissipation functional has a finite
sharp-interface limit.

The $L^2$ Allen--Cahn dynamics do not conserve the spatial average of
$\phi$.  To describe a droplet with prescribed liquid volume, we impose the
global constraint
\[
  \int_\Omega \phi\,\ud^3 x = |\Omega|-2\operatorname{Vol},
\]
where $\operatorname{Vol}$ is the prescribed liquid volume.  Equivalently, the
rate satisfies $\int_\Omega \partial_t\phi\,\ud^3 x=0$, and we introduce a
scalar Lagrange multiplier $\Lambda$ for this constraint in the Onsager
principle. Taking the direction $\tilde\phi=\partial_t\phi$ in the first
variation \eqref{eq:first_var}, and adding the constraint term, gives the
constrained energy rate
\begin{align}
\frac{\ud}{\ud t}E(\phi(t))
-\frac{3\sqrt{2}}{4}\,\Lambda\int_\Omega\partial_t \phi\,\ud^3 x
  &= \frac{3\sqrt{2}}{4}\,\sigma_{\nlg}
     \int_\Omega\left[\frac{1}{\delta}F'(\phi) - \delta\Delta\phi\right]
     \partial_t \phi\,\ud^3 x \notag\\
  &\quad
  + \frac{3\sqrt{2}}{4}\,\sigma_{\nlg}\,\delta
    \int_\Gamma \partial_n\phi\,\partial_t \phi\,\ud^2 x \notag\\
  &\quad
  + (\sigma_{\nsl} - \sigma_{\nsg})\int_\Gamma f_w'(\phi)\partial_t \phi\,\ud^2 x
  - \frac{3\sqrt{2}}{4}\,\Lambda \int_\Omega\partial_t \phi\,\ud^3 x.
  \label{eq:rate}
\end{align}
The phase-field evolution is determined by the constrained Onsager
variational principle:
\begin{equation}\label{eq:Ray_AC}
  \min_{\partial_t\phi,\;\partial_t\phi|_\Gamma}
  \left[
  \frac{\ud}{\ud t}E(\phi(t))
  -\frac{3\sqrt{2}}{4}\,\Lambda\int_\Omega\partial_t\phi\,\ud^3 x
  + {\mathcal D}_{\text{AC}}(\partial_t\phi, \partial_t\phi|_\Gamma)
  \right].
\end{equation}
Taking variations of \eqref{eq:Ray_AC} with respect to
$\partial_t\phi$ and $\partial_t\phi|_\Gamma$, with the energy-rate terms
given by \eqref{eq:rate}, gives the first two equations below. The multiplier
$\Lambda$ is chosen so that variation of the associated constrained Lagrangian
with respect to $\Lambda$ gives
$\int_\Omega\partial_t\phi\,\ud^3 x=0$, which together with the initial condition
$\int_\Omega\phi(\cdot,0)\,\ud^3 x=|\Omega|-2\operatorname{Vol}$ gives the
last equation:
\begin{equation}\label{eq:AC}
  \left\{
  \begin{aligned}
    &\xi\delta\,\partial_t\phi
     = \sigma_{\nlg}\!\left(\delta\Delta\phi - \frac{1}{\delta}F'(\phi)\right) + \Lambda,
     \quad\text{in } \Omega,\\[4pt]
    &\zeta\delta\,\partial_t\phi
     = -\sigma_{\nlg}\delta\,\partial_n\phi
     - (\sigma_{\nsl}-\sigma_{\nsg})\,\frac{4}{3\sqrt{2}}f_w'(\phi),
     \quad\text{on } \Gamma,\\[4pt]
    &\int_\Omega \phi\,\ud^3 x = |\Omega|-2\operatorname{Vol}.
  \end{aligned}
  \right.
\end{equation}
The first equation is the Allen--Cahn equation with a global Lagrange multiplier
$\Lambda$ enforcing mass conservation. The second is a dynamic boundary
condition derived from the wall energy. The third is the integrated form of
the conservation law determined by the initial volume.

\subsection{Sharp-interface limit of the Allen--Cahn model}
\label{ssec:AC_SIL}

We carry out a matched asymptotic expansion, assuming a single connected
moving interface $S_t$ with interface thickness $\delta \ll 1$. For both the
asymptotic analysis and the gradient-flow structure of the sharp-interface
model, we assume that the sharp interface is attached to the substrate. This is
a kinematic condition in the derivation.

\subsubsection{Inner profile and motion by mean curvature}

In the inner region near $S_t$, we introduce the signed distance function
$d(x,t)$ to $S_t$, defined so that $d > 0$ in the gas phase and $d < 0$
in the liquid phase. We seek an ansatz $\phi(x,t) \approx \Phi(d(x,t)/\delta)$,
where $\Phi : \bR \to [-1,1]$ is the one-dimensional transition profile.

The signed distance function satisfies the following properties on $S_t$, where
$d=0$:
\[
\nabla d = n,
\]
where $n$ is the unit normal to $S_t$ pointing from liquid to gas.
Moreover,
\[
\partial_t d = -V.
\]
Indeed, if $d(x(t),t)=0$ on the evolving surface $S_t$, then
\[
\partial_t d+\nabla d\cdot\dot{x}(t)
=\partial_t d+V=0,
\]
where $V$ is the outward normal velocity of $S_t$. Finally,
\[
\Delta d = H,
\]
where $H$ is the mean curvature of $S_t$ with our convention $H=2/R>0$ for a
sphere of radius $R$.

With $r = d/\delta$, the chain rule gives the leading-order approximations
\begin{align}
  \partial_t\phi &\approx \frac{1}{\delta}\Phi'(r)\partial_t d = -\frac{1}{\delta}\Phi'(r)V, \label{eq:inner_t}\\
  \nabla\phi     &\approx \frac{1}{\delta}\Phi'(r)\nabla d = \frac{1}{\delta}\Phi'(r)n, \label{eq:inner_grad}\\
  \Delta\phi     &\approx \frac{1}{\delta^2}\Phi''(r)|\nabla d|^2 + \frac{1}{\delta}\Phi'(r)\Delta d
                  = \frac{1}{\delta^2}\Phi''(r) + \frac{1}{\delta}\Phi'(r)H. \label{eq:inner_lap}
\end{align}
Substituting into the first equation of \eqref{eq:AC} and collecting terms:
\[
  \xi\Phi'V = -\sigma_{\nlg}\!\left(\frac{\Phi'' - F'(\Phi)}{\delta} + \Phi'H\right) - \Lambda.
\]
\textbf{$O(\delta^{-1})$ balance (leading order):}
\begin{equation}\label{eq:ODE_Phi}
  \Phi'' - F'(\Phi) = 0, \quad r \in \bR.
\end{equation}

\textbf{Solving the profile equation.}
Multiply \eqref{eq:ODE_Phi} by $\Phi'$ and integrate:
\begin{equation}\label{eq:first_integral}
  \frac{1}{2}(\Phi')^2 = F(\Phi) + C.
\end{equation}
The boundary conditions $\Phi(\pm\infty) = \pm 1$ with $F(\pm 1) = 0$ force
$C = 0$. Hence
\[
  \Phi' = \frac{1}{\sqrt{2}}(1 - \Phi^2),
\]
which is a separable ODE. Separating variables:
\[
  \int_{0}^{\Phi}\frac{d\phi'}{1-(\phi')^2} = \frac{r}{\sqrt{2}},
\]
giving the unique monotone solution
\begin{equation}\label{eq:tanh}
  \Phi(r) = \tanh\!\left(\frac{r}{\sqrt{2}}\right).
\end{equation}
Near the interface, $\phi(x,t) \approx \tanh\!\left(d(x,t)/(\sqrt{2}\,\delta)\right)$.

\textbf{$O(1)$ balance (sub-leading order):}
Projecting the $O(1)$ equation onto the translational mode $\Phi'$, and
absorbing the resulting constant factor and sign into the spatially constant
multiplier, gives the \emph{motion by mean curvature}:
\begin{equation}\label{eq:MMC}
   \xi V = -\sigma_{\nlg}H + \Lambda.
\end{equation}
The Lagrange multiplier $\Lambda$ enforces global mass conservation and plays
the role of a spatially constant chemical potential. 

\subsubsection{Recovery of the contact line velocity}
\label{sssec:CL_vel}

We now derive the contact line law from the boundary condition in
\eqref{eq:AC}.  Near the contact line, we again use
$\phi \approx \Phi(d(x,t)/\delta)$.  The boundary condition on $\Gamma$ becomes
\[
  \zeta\Phi'\partial_t d = -\sigma_{\nlg}\Phi'\partial_n d
  - (\sigma_{\nsl}-\sigma_{\nsg})\frac{4}{3\sqrt{2}} f_w'(\Phi).
\]
Using \eqref{eq:f_w'}, we have
$\frac{4}{3\sqrt{2}}f_w'(\Phi) = \frac{1}{\sqrt{2}}(\Phi^2-1) = -\Phi'$.
Since $\partial_n d=-\partial_z d$ on $\Gamma$, the equation simplifies to
\[
  \zeta\Phi'\partial_t d = \sigma_{\nlg}\Phi'\partial_z d + (\sigma_{\nsl}-\sigma_{\nsg})\Phi'.
\]
Dividing through by $\Phi' > 0$:
\begin{equation}\label{eq:CL_dist}
  \zeta\partial_t d = \sigma_{\nlg}\partial_z d+\sigma_{\nsl}-\sigma_{\nsg}.
\end{equation}

On the substrate $\Gamma$, the outward normal of $\Omega$ points into the solid,
so $\partial_n d=-\partial_z d$. Since $\nabla d = n$ points from liquid to gas
and makes angle $\theta_{\cl}$ with $\hat z$, we have
\begin{equation}\label{eq:normal_proj}
  \partial_z d = \hat z\cdot\nabla d = \hat z\cdot n = \cos\theta_{\cl}.
\end{equation}
Here $\theta_{\cl}$ is the dynamic contact angle (inner angle between $S_t$
and $\Gamma$).

The time derivative $\partial_t d$ on the substrate at the contact point
$Z(t) = (x(t),y(t),0) \in \partial D_t$
can be related to the contact line speed by differentiating
$d(Z(t), t) \equiv 0$:
\begin{equation}\label{eq:dt_d}
  \partial_t d|_{\partial D_t} + \nabla d\cdot\dot{Z} = 0.
\end{equation}
Since $\dot{Z}=v_{\cl}n_{\cl}$ lies in $\Gamma$ and
$\nabla d\cdot n_{\cl}=\sin\theta_{\cl}$ at the contact point, this gives
$\partial_t d|_{\partial D_t} = -\nabla d\cdot\dot{Z} = -v_{\cl}\sin \theta_{\cl}$.

Substituting \eqref{eq:normal_proj} and
$\partial_t d = -v_{\cl}\sin\theta_{\cl}$
into \eqref{eq:CL_dist}, and using Young's equation:
\begin{equation}\label{eq:CL_law}
\zeta v_{\cl} = \frac{\sigma_{\nlg}(\cos\theta_Y - \cos\theta_{\cl})}{\sin\theta_{\cl}}.
\end{equation}

\begin{rem}\label{remarkS}
The physically natural form of the contact line law is
\begin{equation}
\zeta v_{\cl}
=
\sigma_{\nlg}(\cos\theta_Y - \cos\theta_{\cl}),
\end{equation}
where $\zeta$ is the contact line friction coefficient.

In the phase-field derivation, the evolution law is first obtained in terms
of the normal velocity $V = -\partial_t d$. Using the geometric relation
  $V = v_{\cl}\sin\theta_{\cl}$ leads to an equivalent expression involving
a factor $1/\sin\theta_{\cl}$. Equivalently, the factor $\sin\theta_{\cl}$ can
be absorbed into the effective contact-line friction coefficient, and does not
represent a distinct physical mechanism.
\end{rem}
 
After absorbing the factor $\sin\theta_{\cl}$ into the effective contact-line
friction coefficient, still denoted by $\zeta$, the sharp-interface model
obtained as the sharp-interface limit of the Allen--Cahn model \eqref{eq:AC} is
\begin{equation}\label{eq:SI_AC}
\left\{
\begin{aligned}
&\xi V = -\sigma_{\nlg} H + \Lambda, \quad \text{on } S_t,\\
    &\zeta v_{\cl} = \sigma_{\nlg}(\cos\theta_Y - \cos\theta_{\cl}), \quad \text{on }\pt D_t,\\
    & |\Omega_L(t)|=\operatorname{Vol},\\
    & \pt S_t = \pt D_t.
\end{aligned}
\right. 
\end{equation}
We impose that the capillary surface is attached to the substrate,
$\pt S_t = \pt D_t$. This is a kinematic condition used in the derivation.

\subsection{Variational structure for the sharp-interface system from Allen--Cahn}

The sharp-interface system \eqref{eq:SI_AC} can be derived directly as an
$L^2$-gradient flow of the energy $E(S_t)$ with respect to the dissipation
functional
\begin{equation}\label{eq:diss_SI_AC}
 \sD_{\text{AC}}
  = \frac{\xi}2\int_{S_t} V^2\,\ud\cH^2
  + \frac{\zeta}2 \int_{\partial D_t} v_{\cl}^2\,\ud\cH^1.
\end{equation}
Indeed, minimizing $\frac{\ud E}{\ud t} + \sD_{\text{AC}}$ over $V$ and
$v_{\cl}$, and using Theorem~\ref{thm:first_variation_varifold} after including
the volume constraint, immediately yields \eqref{eq:SI_AC}.

In the graph representation $h(x,y,t)$, $(x,y)\in D_t$, the sharp-interface
model with volume constraint is \cite{GaoLiu21}
\begin{equation}\label{eq:SI_AC_graph}
  \left\{
  \begin{aligned}
    \xi\,\frac{\partial_t h}{\sqrt{1+|\nabla h|^2}}
    &= \sigma_{\nlg}\nabla\cdot\!\left(\frac{\nabla h}{\sqrt{1+|\nabla h|^2}}\right)
      - \Lambda,
      \quad\text{on } D_t,\\[4pt]
      h&=0, \quad\text{on } \pt D_t,\\
    \zeta v_{\cl}
    &= \sigma_{\nlg}\!\left(\cos\theta_Y
      - \frac{1}{\sqrt{1+|\nabla h|^2}}\right),
      \quad\text{on } \partial D_t,\\[4pt]
    \int_{D_t} h\,\ud x\,\ud y &= \operatorname{Vol}.
  \end{aligned}
  \right.
\end{equation}
We refer to
\cite{GL21} for the detailed derivation in the graph representation.

 \section{Cahn--Hilliard Model: Phase-Field and Sharp-Interface}
\label{sec:CH}

In this section, we introduce the phase-field Cahn--Hilliard model through an
$H^{-1}$ bulk dissipation mechanism coupled with the same substrate
dissipation as in the Allen--Cahn model. We then derive its sharp-interface
limit by formal asymptotic analysis. The substrate dynamics are the same as in
the Allen--Cahn case because the leading-order inner profile is again the
solution \eqref{eq:tanh}. The capillary surface dynamics, however, differ from
those of the Allen--Cahn model: instead of local motion by mean curvature, the
Cahn--Hilliard model yields the nonlocal Mullins--Sekerka evolution
\eqref{nonlocalV}, first derived by Pego~\cite{Pego89}. We also present the
variational structures of both the phase-field model and the sharp-interface
Mullins--Sekerka model with contact line dynamics.

\subsection{Phase-field Cahn--Hilliard model}

We consider the same free energy functional $E = E_b + E_w$ as
defined in Section~\ref{ssec:pf_energy}. In contrast to the Allen--Cahn
case, no explicit Lagrange multiplier is introduced here, since mass
conservation is built into the Cahn--Hilliard dynamics through the bulk
$H^{-1}$ metric. We impose the far-field boundary condition
$\partial_n\phi = 0$ or $\phi = 1$ on $\partial\Omega\setminus\Gamma$.

To reflect this conservation law, we introduce the dissipation functional
\begin{equation}\label{eq:diss_CH_full}
  {\mathcal D}_{\text{CH}}(u, u|_\Gamma)
  =\frac{3\sqrt{2}}{4}\,\frac{\xi}2
  \int_\Omega u(-\Delta_N)^{-1}u\,\ud^3 x
  +\frac{3\sqrt{2}}{4}\,\frac{\zeta\delta}{2}
  \int_\Gamma (u|_\Gamma)^2\,\ud^2 x,
\end{equation}
where $u = \partial_t \phi$, and $(-\Delta_N)^{-1}$ denotes the inverse
Neumann Laplacian acting on mean-zero functions.

The first term in \eqref{eq:diss_CH_full} induces an $H^{-1}(\Omega)$ metric in
the bulk, which enforces mass conservation, while the second term is the
$L^2(\Gamma)$ dissipation localized at the substrate.

Recalling the first variation of the free energy from \eqref{eq:first_var}, we
have
\begin{align*}
\frac{\ud}{\ud t}E(\phi(t))
  &= \frac{3\sqrt{2}}{4}\,\sigma_{\nlg}
     \int_\Omega\left[\frac{1}{\delta}F'(\phi) - \delta\Delta\phi\right]
     \partial_t \phi\,\ud^3 x \\
  &\quad
  + \frac{3\sqrt{2}}{4}\,\sigma_{\nlg}\,\delta
    \int_\Gamma \partial_n\phi\,\partial_t \phi\,\ud^2 x \\
  &\quad
  + (\sigma_{\nsl} - \sigma_{\nsg})
    \int_\Gamma f_w'(\phi)\,\partial_t \phi\,\ud^2 x.
\end{align*}

Via the Onsager variational principle,
the evolution of $\phi$ is determined by minimizing the sum of the energy rate
and the dissipation functional:
\begin{equation}\label{Ray1}
  \min_{\partial_t\phi,\;\partial_t\phi|_\Gamma}
  \left[
  \frac{\ud}{\ud t}E(\phi(t))
  + {\mathcal D}_{\text{CH}}(\partial_t\phi, \partial_t\phi|_\Gamma)
  \right].
\end{equation}

\medskip
\noindent

\begin{prop}[Cahn--Hilliard system with contact line dynamics]
The Onsager variational principle yields the following system:
\begin{equation}\label{eq:CH_full}
  \left\{
  \begin{aligned}
    &\xi\,\partial_t\phi = \Delta\mu,
    \quad
    \mu = -\sigma_{\nlg}\delta\Delta\phi
    + \frac{\sigma_{\nlg}}{\delta}F'(\phi),
    \quad\text{in }\Omega,\\[4pt]
    &\zeta\delta\,\partial_t\phi
     = -\sigma_{\nlg}\delta\,\partial_n\phi
     - (\sigma_{\nsl}-\sigma_{\nsg})\,
       \frac{4}{3\sqrt{2}}f_w'(\phi),
     \quad\text{on }\Gamma,\\[4pt]
    &\partial_n\phi = 0 \quad \text{on } \partial\Omega\setminus\Gamma,
    \qquad
    \partial_n\mu = 0 \quad \text{on } \partial\Omega.
  \end{aligned}
  \right.
\end{equation}
\end{prop}

\begin{proof}
We take first variations of \eqref{Ray1} with respect to the bulk rate
$u=\partial_t\phi$ and the substrate rate $u|_\Gamma=\partial_t\phi|_\Gamma$.
Since the Cahn--Hilliard rate has zero mean, let
$\psi=(-\Delta_N)^{-1}u$, so that $-\Delta\psi=u$ with homogeneous Neumann
boundary condition and zero mean. Variation with respect to $u$ gives
\[
  \frac{3\sqrt{2}}{4}\left[
  \sigma_{\nlg}\!\left(\frac{1}{\delta}F'(\phi)-\delta\Delta\phi\right)
  + \xi\psi
  \right]=0.
\]
With
\begin{equation}\label{eq:chem}
\mu = -\sigma_{\nlg}\delta\Delta\phi
    + \frac{\sigma_{\nlg}}{\delta}F'(\phi),
\end{equation}
this stationarity condition is $\mu+\xi\psi=0$. Applying $\Delta$ and using
$-\Delta\psi=\partial_t\phi$ yields
\[
\xi\,\partial_t \phi = \Delta \mu.
\]

The variation with respect to $\partial_t\phi|_\Gamma$ balances the boundary
part of the energy rate with the substrate dissipation, giving
\[
\zeta\delta\,\partial_t\phi
= -\sigma_{\nlg}\delta\,\partial_n\phi
- (\sigma_{\nsl}-\sigma_{\nsg})
\frac{4}{3\sqrt{2}}f_w'(\phi).
\]

The remaining conditions are the natural boundary conditions associated with
the variational formulation.
\end{proof}

Integrating the bulk equation over $\Omega$, we obtain
\[
\xi\frac{\ud}{\ud t}\int_\Omega \phi\,\ud^3 x
=
\int_\Omega \Delta \mu\,\ud^3 x
=
\int_{\partial \Omega} \partial_n \mu\,\ud^2 x
=
0,
\]
using the Neumann boundary condition $\partial_n \mu = 0$. Thus the Cahn--Hilliard dynamics conserve the total mass.

The system \eqref{eq:CH_full} can be interpreted as follows. The bulk equation
\[
\partial_t \phi = \frac{1}{\xi}\Delta \mu
\]
describes diffusion driven by gradients of the chemical potential. The chemical potential $\mu$ contains both interfacial effects through $\Delta\phi$ and bulk thermodynamic effects through $F'(\phi)$. The boundary condition on $\Gamma$ models contact-line relaxation caused by an imbalance of surface tensions.

\subsection{Sharp-interface limit of the Cahn--Hilliard system}
\label{ssec:CH_SIL_full}
We now derive the sharp-interface limit from the Cahn--Hilliard system
\eqref{eq:CH_full}. This asymptotic analysis follows the classical result of
Pego~\cite{Pego89}. We present the formal derivation and clarify the
intermediate steps.
We assume a well-prepared initial condition so that a diffuse interface of thickness $O(\delta)$ persists for short times.

\begin{thm}[Mullins--Sekerka limit]\label{thm4.2}
In the limit $\delta \to 0$, the Cahn--Hilliard system converges formally to
the Mullins--Sekerka problem
\begin{equation}\label{nonlocalV}
  \left\{
  \begin{aligned}
    &\Delta \mu = 0, \quad &&\text{in } \Omega_G(t) \text{ and }\Omega_L(t),\\
    &\pt_n \mu =0, \quad &&\text{on } \pt\Omega,\\
    &\mu_L = \mu_G = -\frac{\sqrt{2}}{3}\,\sigma_{\nlg}\,H,
      \quad &&\text{on } S_t, \\
    &2\xi V = \partial_n \mu_L - \partial_n \mu_G,
      \quad &&\text{on } S_t,
  \end{aligned}
  \right.
\end{equation}
where the chemical potential is harmonic in each bulk phase, and the Gibbs--Thomson relation
$
\mu_L = \mu_G = -\frac{\sqrt{2}}{3}\,\sigma_{\nlg}\,H
$
holds on the interface. The mean curvature $H$ is defined so that a sphere of
radius $R$ with outward unit normal has $H=2/R>0$. The normal derivatives in
the velocity law are taken in the direction $n$ pointing from liquid to gas.
\end{thm}

\begin{rem}
Formula \eqref{nonlocalV} agrees with equation (1.4) in
\cite{AlikakosBatesChen} after the change of variable $v=-\mu$. The sign
difference from equations (2.6) and (5.12) in \cite{Pego89} comes from the
definition of the normal velocity $V$.
\end{rem}

\begin{proof}
First, using asymptotic expansion, we compute the \emph{outer expansion} in the
bulk regions.

Away from the interface, we expand
\[
\phi = \phi_0 + \delta \phi_1 + \cdots,
\qquad
\mu = \mu_0 + \delta \mu_1 + \cdots.
\]

The singular term in the chemical potential forces the well-prepared outer
states into the wells of the potential:
\[
\phi_0 = \pm 1.
\]
Substituting into \eqref{eq:CH_full}, we obtain
\[
\Delta \mu_0 = 0,
\]
so $\mu_0$ is harmonic in each phase.
This means the bulk dynamics are quasi-static diffusion, with all nontrivial dynamics localized at the interface.

Next, we compute the \emph{inner expansion} in the interfacial layer.

Near the interface $S_t$, we introduce the signed distance function $d(x,t)$ and the stretched variable
\[
r = \frac{d(x,t)}{\delta}.
\]
We seek an inner approximation
\[
\phi(x,t) \approx \Phi(r),
\]
where $\Phi$ is the transition profile.

Recall the geometric identities on $S_t$ give
\begin{equation}\label{eq:GeoId}
\nabla d = n, \quad
\partial_t d = -V, \quad
\Delta d = H.  
\end{equation}
Thus,
\begin{align}
\partial_t \phi &\approx -\frac{1}{\delta} \Phi'(r) V, \\
\nabla \phi &\approx \frac{1}{\delta} \Phi'(r) n, \\
\Delta \phi &\approx \frac{1}{\delta^2} \Phi''(r) + \frac{1}{\delta} \Phi'(r) H.
\end{align}

The chemical potential \eqref{eq:chem} becomes
\begin{equation}
\mu \approx -\sigma_{\nlg}
\left[
\frac{1}{\delta}(\Phi'' - F'(\Phi)) + \Phi' H
\right].
\end{equation}

\textbf{$O(\delta^{-1})$ balance (leading order):}
At order $O(\delta^{-1})$ in the chemical potential, we obtain
\begin{equation}\label{eq:phi}
\Phi'' - F'(\Phi) = 0,  
\end{equation}
whose unique monotone solution is
\[
\Phi(r) = \tanh\!\left(\frac{r}{\sqrt{2}}\right).
\]

\textbf{$O(1)$ balance (sub-leading order):}

At the next order, 
\[
\phi(x,t) \approx \Phi(r)+ \delta U_1(r),
\]
and we obtain a linearized problem for $U_1(r)$:
\begin{equation}\label{eq:U1}
\mathcal{L} U_1 = f(r),
\qquad
\mathcal{L} = -\partial_{rr} + F''(\Phi),\,\, f(r) = H \Phi'(r) + \frac{\mu_0}{\sigma_{\nlg}}.
\end{equation}
 
From \eqref{eq:phi}, we have $\mathcal{L}\Phi' = 0$. Solvability for \eqref{eq:U1} requires
\[
\int_{-\infty}^{\infty} f(r)\Phi'(r)\,dr = 0,
\]
which gives
\[
H \int_{-\infty}^{\infty} (\Phi')^2 dr
+ \frac{\mu_0}{\sigma_{\nlg}}
\int_{-\infty}^{\infty} \Phi' dr = 0.
\]
Using
\[
\int_{-\infty}^{\infty} \Phi' dr = 2, 
\qquad 
\int_{-\infty}^{\infty} (\Phi')^2 dr = \frac{2\sqrt{2}}{3},
\]
we conclude
\begin{equation}
\mu_0 =  -\frac{\sqrt{2}}{3}\,\sigma_{\nlg} H.
\end{equation}
Thus curvature shifts the chemical potential, producing the Gibbs--Thomson
effect.

\paragraph{Interface velocity law.}

The Cahn--Hilliard equation reads
\[
\partial_t \phi = \frac{1}{\xi} \Delta \mu.
\]

In the inner region,
\[
\partial_t \phi \approx -\frac{1}{\delta} \Phi'(r) V,
\qquad
\Delta \mu \approx \partial_{nn}\mu.
\]
Integrating the equation $\xi\partial_t\phi=\Delta\mu$ across a thin tubular
neighborhood from the liquid side to the gas side gives
\[
\xi\int_{-\infty}^{\infty}
\left(-\frac{1}{\delta}\Phi'(r)V\right)\delta\,\ud r
= \partial_n\mu_G-\partial_n\mu_L.
\]
Since $\int_{-\infty}^{\infty}\Phi'\,\ud r=2$, we obtain
\begin{equation}
2\xi V = \partial_n\mu_L - \partial_n\mu_G.
\end{equation}
This means the interface moves by diffusive flux imbalance across it.

\end{proof}

\subsubsection{Contact line dynamics}

The analysis near the contact line follows the same structure as in the
Allen--Cahn case in Section~\ref{sssec:CL_vel}.

The boundary condition reduces to
\[
\zeta v_{\cl}
=
\frac{\sigma_{\nlg}(\cos\theta_Y - \cos\theta_{\cl})}
{\sin\theta_{\cl}},
\quad \text{on } \partial D_t.
\]
As in Remark~\ref{remarkS}, absorbing the factor $\sin\theta_{\cl}$ into
$\zeta$ gives
\[
\zeta v_{\cl}
=
\sigma_{\nlg}(\cos\theta_Y - \cos\theta_{\cl}),
\quad \text{on } \partial D_t.
\]
\subsubsection{The Mullins--Sekerka problem with contact line dynamics}

Based on the preceding asymptotic analysis and the contact-line calculation, we
summarize the sharp-interface Mullins--Sekerka problem with moving contact
line:
\begin{equation}\label{MScl}
  \left\{
  \begin{aligned}
    &\Delta \mu_G = 0, \quad &&\text{in } \Omega_G(t),\\
    &\Delta \mu_L = 0, \quad &&\text{in } \Omega_L(t), \\
    &\pt_n \mu =0, \quad &&\text{on } \pt\Omega,\\
    &\mu_L = \mu_G = -\frac{\sqrt{2}}{3}\,\sigma_{\nlg}\,H,
      \quad &&\text{on } S_t, \\
    &2\xi V = \partial_n \mu_L - \partial_n \mu_G,
      \quad &&\text{on } S_t, \\
    &\zeta v_{\cl}
      = \sigma_{\nlg}(\cos\theta_Y - \cos\theta_{\cl}),
      \quad &&\text{on } \partial D_t,
      \\
      &\pt S_t = \pt D_t.
  \end{aligned}
  \right.
\end{equation}
We impose that the capillary surface is attached to the substrate,
$\pt S_t = \pt D_t$.

The contact-line equation describes the motion of $\partial D_t$, where
$v_{\cl}$ is the contact line velocity and $\theta_{\cl}$ is the dynamic
contact angle. It reflects the balance between capillary forces and
contact-line dissipation.

\subsection{Variational formulation of the Mullins--Sekerka problem}

We now present a variational interpretation of the Mullins--Sekerka dynamics
coupled with the moving contact line.

Recall that the rate of change of the free energy is given by
\begin{equation}
  \frac{\ud E}{\ud t}
  = \sigma_{\nlg} \int_{S_t} H V \,\ud \cH^2
  + \sigma_{\nlg} \int_{\partial D_t}
    (\cos\theta_{\cl} - \cos\theta_Y)\, v_{\cl} \,\ud \cH^1.
\end{equation}

We next define the dissipation functional through an auxiliary problem for
$\mu$. For a given interface $S_t$ and normal velocity $V$, we introduce a
chemical potential $\mu$ defined on $\Omega_G(t)\cup\Omega_L(t)$ as the
solution to
\begin{equation}\label{tm415mu}
\begin{cases}
\Delta \mu = 0, & \text{in } \Omega_G(t)\cup\Omega_L(t), \\
\partial_n \mu = 0, & \text{on } \partial \Omega, \\
\pt_n \mu_L - \pt_n \mu_G = 2\xi V, & \text{on } S_t, \\
[\mu] = 0, & \text{on } S_t.
\end{cases}
\end{equation}

Define the dissipation functional associated with the Mullins--Sekerka dynamics as
\begin{equation}
  {\mathcal D}_{\text{MS}}
  = \frac{\beta}2 \int_{\Omega_G(t)\cup\Omega_L(t)} |\nabla \mu|^2 \,\ud x
  + \frac{\zeta}2 \int_{\partial D_t} v_{\cl}^2 \,\ud \cH^1,
\end{equation}
where the mobility constant $\beta$ is to be determined.
The first term represents $H^{-1}$-type bulk diffusion dissipation in terms of the chemical potential $\mu$, while the second term accounts for dissipation localized at the contact line. 

Below, we represent the $H^{-1}$-type bulk dissipation as a dual pairing of the
chemical potential and the normal velocity, where $\mu$ is solved from
\eqref{tm415mu}. Using Green's identity, $\Delta \mu=0$, and the transmission
conditions, we obtain
\begin{equation}
\int_{\Omega_G(t)\cup\Omega_L(t)} |\nabla \mu|^2 \,\ud x
= \int_{S_t} (\mu_L \partial_n \mu_L-\mu_G \partial_n \mu_G) \,\ud \cH^2
=  2\xi \int_{S_t} \mu V \,\ud \cH^2.
\end{equation}
This identity connects the bulk dissipation to the interfacial motion.

We determine the evolution by minimizing the functional
\[
\frac{\ud E}{\ud t} +  {\mathcal D}_{\text{MS}}
\]
with respect to the interface velocity $V$ and the contact line velocity
$v_{\cl}$. Taking variations in $V$ and $v_{\cl}$ yields
\begin{equation}\label{tm418}
    \begin{aligned}
&\sigma_{\nlg} H = -2 \xi \beta \mu,
\quad \text{on } S_t,\\
&\zeta v_{\cl}
= \sigma_{\nlg}(\cos\theta_Y - \cos\theta_{\cl}),
\quad \text{on } \partial D_t.
\end{aligned}
\end{equation}

Finally, we choose the mobility in the dissipation functional to recover the
same coefficients as in the sharp-interface limit of the Cahn--Hilliard model.
Choosing
\[
\beta = \frac{3}{2\sqrt{2}\,\xi}
\]
recovers the Gibbs--Thomson relation
\[
\mu = -\frac{\sqrt{2}}{3}\,\sigma_{\nlg}\,H.
\]
Combining \eqref{tm415mu} and \eqref{tm418} yields the full sharp-interface
Mullins--Sekerka system with moving contact line \eqref{MScl}.

 \section{Numerical Schemes and Results}
\label{sec:numerics}

We now construct numerical schemes that preserve the variational structure of
the phase-field models.  Recall from Section~\ref{ssec:pf_energy} that the
total free energy is
\begin{equation}\label{eq:E_num}
  E(\phi)=E_b(\phi)+E_w(\phi)
  =
  \frac{3\sqrt2}{4}\sigma_{\nlg}
  \int_\Omega
  \left[\frac{\delta}{2}|\nabla\phi|^2+\frac{1}{\delta}F(\phi)\right]
  \,\ud^3 x
  +(\sigma_{\nsl}-\sigma_{\nsg})\int_\Gamma f_w(\phi)\,\ud^2 x .
\end{equation}
To handle numerical values of $\phi$ outside $[-1,1]$, we use the following
$C^2$ extension of the wall energy density in \eqref{eq:E_num}:
\begin{equation}\label{eq:f_extended}
  f_w(\phi)
  =
  \begin{cases}
    \displaystyle \frac{3}{4}(\phi-1)^2,
    & \phi>1, \\[0.4em]
    \displaystyle \frac{\phi^3-3\phi}{4}+ \frac{1}{2},
    & -1\le\phi\le 1, \\[0.4em]
    \displaystyle 1-\frac{3}{4}(\phi+1)^2,
    & \phi<-1 .
  \end{cases}
\end{equation}
This extension agrees with \eqref{eq:GLw} on $[-1,1]$ and satisfies
$|f_w''(\phi)|\le 3/2$ for all $\phi\in\bR$.  The double-well potential needs
no analogous extension, since $F(\phi)=\frac14(\phi^2-1)^2$ is already defined
on all of $\bR$ and satisfies $F''(\phi)\ge -1$.

The discretization is introduced in two steps.  We first apply a minimizing
movement scheme in time using \eqref{eq:E_num} and the Onsager dissipation
functionals from Sections~\ref{sec:AC} and~\ref{sec:CH}.  We then discretize
the resulting variational problems in space by conforming finite elements.
Thus the Allen--Cahn scheme is the minimizing movement scheme for
\eqref{eq:Ray_AC}, while the Cahn--Hilliard scheme is the minimizing movement
scheme for \eqref{Ray1}.

\subsection{Minimizing movement (JKO) scheme}

Let $t_n=n\Delta t$.  Given $\phi^{n-1}$, the Allen--Cahn update is obtained by
minimizing the free energy plus the time-discrete Allen--Cahn dissipation:
\begin{equation}\label{eq:AC_JKO}
  \phi^n
  =
  \argmin_{\substack{\phi\in H^1(\Omega)\\
  \int_\Omega\phi\,\ud^3 x=\int_\Omega\phi^{n-1}\,\ud^3 x}}
  \left\{
  E(\phi)
  +\Delta t\,{\mathcal D}_{\text{AC}}\!\left(
  \frac{\phi-\phi^{n-1}}{\Delta t},
  \left.\frac{\phi-\phi^{n-1}}{\Delta t}\right|_\Gamma
  \right)
  \right\}.
\end{equation}
The constraint is the time-discrete form of the conservative Allen--Cahn volume
constraint.  If the initial data satisfy
$\int_\Omega\phi^0\,\ud^3 x=|\Omega|-2\operatorname{Vol}$, then
\eqref{eq:AC_JKO} preserves this value at every time step.

The Cahn--Hilliard minimizing movement step is
\begin{equation}\label{eq:CH_JKO}
  \phi^n
  =
  \argmin_{\phi\in H^1(\Omega)}
  \left\{
  E(\phi)
  +\Delta t\,{\mathcal D}_{\text{CH}}\!\left(
  \frac{\phi-\phi^{n-1}}{\Delta t},
  \left.\frac{\phi-\phi^{n-1}}{\Delta t}\right|_\Gamma
  \right)
  \right\}.
\end{equation}

\begin{thm}[Energy stability and conditional well-posedness]
\label{thm:JKO_energy}
The minimizers of \eqref{eq:AC_JKO} and \eqref{eq:CH_JKO}, whenever they
exist, satisfy
\begin{align}
  E(\phi^n)
  &+\Delta t\,{\mathcal D}_{\text{AC}}\!\left(
  \frac{\phi^n-\phi^{n-1}}{\Delta t},
  \left.\frac{\phi^n-\phi^{n-1}}{\Delta t}\right|_\Gamma
  \right)
  \le E(\phi^{n-1})
  \label{eq:AC_energy_law}\\
  E(\phi^n)
  &+\Delta t\,{\mathcal D}_{\text{CH}}\!\left(
  \frac{\phi^n-\phi^{n-1}}{\Delta t},
  \left.\frac{\phi^n-\phi^{n-1}}{\Delta t}\right|_\Gamma
  \right)
  \le E(\phi^{n-1}).
  \label{eq:CH_energy_law}
\end{align}
In addition, consider the Allen--Cahn minimizing movement restricted to a
finite-dimensional space, with the same mass constraint as in
\eqref{eq:AC_JKO}.  With the extended wall energy density
\eqref{eq:f_extended}, if
\begin{equation}\label{eq:AC_convex_dt}
  \Delta t
  \le
  \min\!\left\{
  \frac{\xi\delta^2}{\sigma_{\nlg}},
  \frac{\zeta\delta}{\sqrt2|\sigma_{\nsl}-\sigma_{\nsg}|}
  \right\},
\end{equation}
with the second restriction omitted when $\sigma_{\nsl}=\sigma_{\nsg}$, then
the Allen--Cahn objective is strictly convex on the admissible affine space.
Consequently the discrete Allen--Cahn minimizing movement has a unique
minimizer.

Similarly, the finite-dimensional Cahn--Hilliard minimizing movement, with the
$H^{-1}$ term acting on mean-zero increments, is strictly convex on the
admissible affine space if
\begin{equation}\label{eq:CH_convex_dt}
  \Delta t
  <
  \min\!\left\{
  \frac{4\xi\delta^3}{\sigma_{\nlg}},
  \frac{\zeta\delta}{\sqrt2|\sigma_{\nsl}-\sigma_{\nsg}|}
  \right\},
\end{equation}
again omitting the second restriction when $\sigma_{\nsl}=\sigma_{\nsg}$.
Consequently the discrete Cahn--Hilliard minimizing movement has a unique
minimizer.
\end{thm}

\begin{proof}
For the energy laws, the previous state $\phi^{n-1}$ belongs to the admissible
class of each minimizing movement.  Evaluating the objective at
$\phi^{n-1}$ gives the right-hand side, and minimality gives
\eqref{eq:AC_energy_law} and \eqref{eq:CH_energy_law}.

Next we justify the convexity assertion for the Allen--Cahn step.  Let
\[
  J_{\text{AC}}(\phi;\phi^{n-1})
  :=
  E(\phi)
  +\Delta t\,{\mathcal D}_{\text{AC}}\!\left(
  \frac{\phi-\phi^{n-1}}{\Delta t},
  \left.\frac{\phi-\phi^{n-1}}{\Delta t}\right|_\Gamma
  \right),
\]
restricted to the finite-dimensional space.  For an admissible variation
$v_h$ satisfying $(v_h,1)_\Omega=0$, its second variation is
\begin{align*}
  J_{\text{AC}}''(\phi;\phi^{n-1})[v_h,v_h]
  &=
  \frac{3\sqrt2}{4}\sigma_{\nlg}\delta
  \|\nabla v_h\|_{L^2(\Omega)}^2 \\
  &\quad
  +\frac{3\sqrt2}{4}
  \int_\Omega
  \left[
  \frac{\sigma_{\nlg}}{\delta}F''(\phi)
  +\frac{\xi\delta}{\Delta t}
  \right]v_h^2\,\ud^3 x \\
  &\quad
  +\int_\Gamma
  \left[
  (\sigma_{\nsl}-\sigma_{\nsg})f_w''(\phi)
  +\frac{3\sqrt2}{4}\frac{\zeta\delta}{\Delta t}
  \right]v_h^2\,\ud^2 x .
\end{align*}
Since $F''(\phi)=3\phi^2-1\ge -1$ and the extended wall energy satisfies
$|f_w''|\le 3/2$, the bulk and boundary zeroth-order coefficients are
nonnegative under \eqref{eq:AC_convex_dt}.  The
remaining gradient term is strictly positive for every nonzero admissible
variation, because the mass constraint excludes nonzero constants.  Hence
$J_{\text{AC}}$ is strictly convex on the admissible affine space.  Coercivity
follows from the quartic bulk energy, the gradient term, and the quadratic
movement cost.  Therefore the finite-dimensional constrained minimization
problem has a unique minimizer by the direct method.

For the Cahn--Hilliard step, the only new point is the bulk metric.  For a
mean-zero admissible variation $v_h$, the bulk part of the second variation is
bounded from below by
\[
  \frac{3\sqrt2}{4}
  \left[
  \sigma_{\nlg}\delta\|\nabla v_h\|_{L^2(\Omega)}^2
  -\frac{\sigma_{\nlg}}{\delta}\|v_h\|_{L^2(\Omega)}^2
  +\frac{\xi}{\Delta t}\|v_h\|_{H^{-1}_N(\Omega)}^2
  \right].
\]
Using the standard Neumann interpolation estimate
$\|v_h\|_{L^2(\Omega)}^2
\le \|\nabla v_h\|_{L^2(\Omega)}\|v_h\|_{H^{-1}_N(\Omega)}$,
this lower bound is positive for every nonzero $v_h$ if
$\Delta t<4\xi\delta^3/\sigma_{\nlg}$.  The boundary part is nonnegative under
the same wall restriction as above, because $|f_w''|\le3/2$.  Thus the
Cahn--Hilliard objective is strictly convex.  Coercivity and finite-dimensional
compactness then give existence, and strict convexity gives uniqueness.
\end{proof}

\begin{rem}[Convex-splitting variant]\label{rem:convex_splitting}
An unconditionally solvable variant is obtained by replacing $E(\phi)$ in the
minimizing movement objectives by a convex-splitting linearization
$E_{\rm cs}(\phi;\phi^{n-1})$.  For the double-well potential, write
\[
  F(\phi)=F_c(\phi)+F_e(\phi),
  \qquad
  F_c(\phi)=\frac14\phi^4+\frac14,\qquad
  F_e(\phi)=-\frac12\phi^2 .
\]
The concave part is then replaced by its tangent at the previous time step:
\[
  F(\phi)
  \quad\leadsto\quad
  F_c(\phi)+F_e(\phi^{n-1})
  +F_e'(\phi^{n-1})(\phi-\phi^{n-1}) .
\]
For the wall energy, one may apply the same idea to
$g(\phi)=(\sigma_{\nsl}-\sigma_{\nsg})f_w(\phi)$.  Since
$|f_w''|\le3/2$, choose
$L_\Gamma\ge \frac32|\sigma_{\nsl}-\sigma_{\nsg}|$ and write
\[
  g(\phi)=\left(g(\phi)+\frac{L_\Gamma}{2}\phi^2\right)
  -\frac{L_\Gamma}{2}\phi^2 .
\]
Linearizing the last concave quadratic gives a functional
$E_{\rm cs}(\phi;\phi^{n-1})$ that is convex in $\phi$ and agrees with
$E(\phi)$ at $\phi=\phi^{n-1}$.  The corresponding finite-dimensional
minimizing movement objective is therefore strictly convex for every
$\Delta t>0$, because the quadratic movement term and the gradient term exclude
nonzero null directions on the admissible mass class.  Hence the convex-split
scheme has a unique minimizer without a time-step restriction.  In the
computations reported below, however, we use the minimizing movement schemes
with the original energy $E(\phi)$.
\end{rem}

\subsection{Finite element discretization}

Let $\mathcal T_h$ be a shape-regular mesh of $\Omega$ that resolves
the substrate $\Gamma$.  For the two-dimensional computations below, we use a
uniform rectangular mesh and the $H^1$-conforming tensor-product finite element
space
\[
  V_h=\{v_h\in C^0(\overline\Omega):v_h|_K\in\mathbb Q_k(K)
  \text{ for }K\in\mathcal T_h\}\subset H^1(\Omega),
\]
where $\mathbb Q_k(K)$ is the space of polynomials of degree at most $k$ in
each variable on the element $K$.  The three-dimensional computations use the
analogous tensor-product space on hexahedral elements.  We also define the
mean-zero subspace
\[
  V_{h,0}:=\{v_h\in V_h:(v_h,1)_\Omega=0\}.
\]

We write $(u,v)_\Omega=\int_\Omega uv\,\ud x$ and
$\langle u,v\rangle_\Gamma=\int_\Gamma uv\,\ud s$.  The finite element energy
$E_h$ is obtained by restricting the original energy $E$ to $V_h$ and
evaluating the integrals by the quadrature used in the computation.  For a
finite element rate $r_h\in V_h$, set
\[
  {\mathcal D}_{\text{AC},h}(r_h,r_h|_\Gamma)
  =
  \frac{3\sqrt2}{4}
  \left[
  \frac{\xi\delta}{2}\|r_h\|_{L^2(\Omega)}^2
  +\frac{\zeta\delta}{2}\|r_h\|_{L^2(\Gamma)}^2
  \right].
\]

The finite element Allen--Cahn minimizing movement is
\begin{equation}\label{eq:AC_JKO_h}
  \phi_h^n
  =
  \argmin_{\substack{\phi_h\in V_h\\
  (\phi_h,1)_\Omega=(\phi_h^{n-1},1)_\Omega}}
  \left\{
  E_h(\phi_h)
  +\Delta t\,{\mathcal D}_{\text{AC},h}\!\left(
  \frac{\phi_h-\phi_h^{n-1}}{\Delta t},
  \left.\frac{\phi_h-\phi_h^{n-1}}{\Delta t}\right|_\Gamma
  \right)
  \right\}.
\end{equation}
Introducing the Lagrange multiplier $\Lambda_h^n\in\bR$ and using the
definitions of $E_h$ and ${\mathcal D}_{\text{AC},h}$, we obtain the following
discrete Euler--Lagrange equations: find
$(\phi_h^n,\Lambda_h^n)\in V_h\times\bR$ such that
\begin{subequations}\label{eq:AC_FEM}
\begin{align}
  &\left(\frac{\xi\delta(\phi_h^n-\phi_h^{n-1})}{\Delta t},\psi_h\right)_\Omega
  +\sigma_{\nlg}\delta(\nabla\phi_h^n,\nabla\psi_h)_\Omega
  +\left(\frac{\sigma_{\nlg}}{\delta}F'(\phi_h^n),\psi_h\right)_\Omega
  -(\Lambda_h^n,\psi_h)_\Omega \notag\\
  &\quad
  +\left\langle
  \frac{\zeta\delta(\phi_h^n-\phi_h^{n-1})}{\Delta t}
  +\frac{4}{3\sqrt2}(\sigma_{\nsl}-\sigma_{\nsg})f_w'(\phi_h^n),
  \psi_h
  \right\rangle_\Gamma
  =0,
  \qquad \forall \psi_h\in V_h, \label{eq:AC_FEM_a}\\
  &(\phi_h^n,1)_\Omega=(\phi_h^{n-1},1)_\Omega. \label{eq:AC_FEM_b}
\end{align}
\end{subequations}
This is a backward-Euler variational discretization of \eqref{eq:AC}.  In
particular, the boundary contribution in \eqref{eq:AC_FEM_a} is the weak form
of the dynamic boundary condition
\[
  \zeta\delta\frac{\phi_h^n-\phi_h^{n-1}}{\Delta t}
  =
  -\sigma_{\nlg}\delta\,\partial_n\phi_h^n
  -\frac{4}{3\sqrt2}(\sigma_{\nsl}-\sigma_{\nsg})f_w'(\phi_h^n)
  \quad\text{on }\Gamma.
\]

For the Cahn--Hilliard scheme, define the discrete Neumann $H^{-1}$ norm for
mean-zero $r_h\in V_{h,0}$ by
\begin{equation}\label{eq:discrete_Hminus1}
  \|r_h\|_{H^{-1}_{N,h}(\Omega)}^2
  :=
  (\nabla w_h,\nabla w_h)_\Omega,
  \qquad \text{where } w_h\in V_{h,0}\text{ solves }
  (\nabla w_h,\nabla\eta_h)_\Omega=(r_h,\eta_h)_\Omega,
  \quad \forall \eta_h\in V_{h,0}.
\end{equation}
The finite element Cahn--Hilliard dissipation is
obtained from \eqref{eq:diss_CH_full} by replacing the Neumann inverse with the
discrete Neumann inverse: for $r_h\in V_{h,0}$,
\[
  {\mathcal D}_{\text{CH},h}(r_h,r_h|_\Gamma)
  =
  \frac{3\sqrt2}{4}
  \left[
  \frac{\xi}{2}\|r_h\|_{H^{-1}_{N,h}(\Omega)}^2
  +\frac{\zeta\delta}{2}\|r_h\|_{L^2(\Gamma)}^2
  \right].
\]
The discrete Cahn--Hilliard minimizing movement is
\begin{equation}\label{eq:CH_JKO_h}
  \phi_h^n
  =
  \argmin_{\phi_h\in V_h}
  \left\{
  E_h(\phi_h)
  +\Delta t\,{\mathcal D}_{\text{CH},h}\!\left(
  \frac{\phi_h-\phi_h^{n-1}}{\Delta t},
  \left.\frac{\phi_h-\phi_h^{n-1}}{\Delta t}\right|_\Gamma
  \right)
  \right\}.
\end{equation}
The dissipation is finite only when
$\phi_h-\phi_h^{n-1}\in V_{h,0}$, so the mass constraint is implicit in the
$H^{-1}_{N,h}$ metric.  Introducing the chemical potential gives the mixed
Euler--Lagrange form: find $(\phi_h^n,\mu_h^n)\in V_h\times V_h$ such that
\begin{subequations}\label{eq:CH_FEM_h}
\begin{align}
  &\left(\frac{\xi(\phi_h^n-\phi_h^{n-1})}{\Delta t},\eta_h\right)_\Omega
  +(\nabla\mu_h^n,\nabla\eta_h)_\Omega=0,
  \qquad \forall \eta_h\in V_h, \label{eq:CH_FEM_h_a}\\
  &-(\mu_h^n,\psi_h)_\Omega
  +\sigma_{\nlg}\delta(\nabla\phi_h^n,\nabla\psi_h)_\Omega
  +\left(\frac{\sigma_{\nlg}}{\delta}F'(\phi_h^n),\psi_h\right)_\Omega
  \notag\\
  &\quad
  +\left\langle
  \frac{\zeta\delta(\phi_h^n-\phi_h^{n-1})}{\Delta t}
  +\frac{4}{3\sqrt2}(\sigma_{\nsl}-\sigma_{\nsg})f_w'(\phi_h^n),
  \psi_h
  \right\rangle_\Gamma
  =0,
  \qquad \forall \psi_h\in V_h. \label{eq:CH_FEM_h_b}
\end{align}
\end{subequations}

The fully discrete schemes inherit the variational properties of the
minimizing movements in Theorem~\ref{thm:JKO_energy}.  In particular, exact
minimizers of \eqref{eq:AC_JKO_h} and \eqref{eq:CH_JKO_h} satisfy the
corresponding discrete energy inequalities with $E$,
${\mathcal D}_{\text{AC}}$, and ${\mathcal D}_{\text{CH}}$ replaced by $E_h$,
${\mathcal D}_{\text{AC},h}$, and ${\mathcal D}_{\text{CH},h}$.  The
Allen--Cahn scheme conserves mass by the constraint \eqref{eq:AC_FEM_b}, while
the Cahn--Hilliard scheme conserves mass by taking $\eta_h=1$ in
\eqref{eq:CH_FEM_h_a}.  The same finite-dimensional convexity argument gives
conditional uniqueness under the time-step restrictions
\eqref{eq:AC_convex_dt} and \eqref{eq:CH_convex_dt}.

\subsection{Implementation details and linear system solvers}

The computations below were carried out with an MFEM implementation
\cite{MFEM} of the fully discrete schemes \eqref{eq:AC_FEM} and
\eqref{eq:CH_FEM_h}.  For the Allen--Cahn model, the code is written
with MPI support and is used in both two and three dimensions.  This is possible
because the linearized Allen--Cahn system admits a relatively scalable solver.
The Cahn--Hilliard model leads to a mixed system for $\phi_h$ and $\mu_h$, whose
linearized system is much harder to precondition.  In the computations reported
here, the Cahn--Hilliard solver is therefore used only in two dimensions, and
only in serial, with a UMFPACK direct solver.  We leave the development of a
scalable Cahn--Hilliard solver for future work.

For the Allen--Cahn scheme, the nonlinear residual in
\eqref{eq:AC_FEM_a}--\eqref{eq:AC_FEM_b} is assembled directly, including the
scalar Lagrange multiplier for the mass constraint.  At each Newton step, let
$A$ be the phase-field Jacobian with respect to $\phi_h$, let $m$ be the
assembled mass vector, and let $R_\phi$ and $R_c$ be the phase-field and
constraint residuals.  Instead of solving the full indefinite augmented system,
we eliminate the scalar multiplier by a one-dimensional Schur complement.  The
code solves
\[
  A u=-R_\phi,\qquad A q=m,
\]
and then sets
\[
  \delta\Lambda
  =
  \frac{-R_c-m^T u}{m^T q},
  \qquad
  \delta\phi=u+\delta\Lambda\,q .
\]
Thus each Newton iteration requires two solves with the same phase-field
Jacobian.  These solves are performed by a preconditioned GMRES solver with the
hypre BoomerAMG \cite{BoomerAMG} preconditioner.  The AMG preconditioner is
cached across linear solves and time steps.  It is rebuilt at the first solve,
after a rejected time step, or if the GMRES iteration count is larger than $30$.

For the current Cahn--Hilliard solver, Newton linearization of
\eqref{eq:CH_FEM_h_a}--\eqref{eq:CH_FEM_h_b} gives a coupled two-by-two block
system for the increments $(\delta\phi,\delta\mu)$.  This linearized system is
solved using the direct UMFPACK solver.  This gives a reliable serial
implementation for the two-dimensional tests below.  Since this direct solver
is not scalable, we do not report three-dimensional Cahn--Hilliard computations
here.

Both solvers use the same adaptive time-stepping safeguards.
We start with a target time step size $\Delta t$.
If Newton fails to converge for a given time step size $\Delta t$, the trial
state is discarded and the step is reduced by half to $\Delta t/2$, down to the
minimum allowed step size $\Delta t_{\min}$.  After an accepted shortened step,
the next proposed step is increased by at most a factor of two until the target
step size is recovered.
The Newton tolerance is $10^{-10}$ with at most $20$ iterations per time step.

\subsection{Numerical examples}

We now illustrate the fully discrete schemes with two wetting and dewetting
tests.  The examples are chosen to test the two features emphasized in the
analysis: recovery of the Young-angle equilibrium determined by the wall
energy, and the different transient behavior induced by the $L^2$ and
$H^{-1}$ bulk metrics.  Unless otherwise stated, we set
$\sigma_{\nlg}=1$, $\xi=\zeta=1$, and
$\sigma_{\nsl}-\sigma_{\nsg}=-\sigma_{\nlg}\cos\theta_Y$, so that the
prescribed parameter $\theta_Y$ is the Young angle in
\eqref{eq:Young_law}.

\subsubsection{Two-dimensional wetting and dewetting.}
The two-dimensional test uses a homogeneous planar substrate, for which the
sharp-interface equilibrium is a circular cap with contact angle
$\theta_Y$.  By symmetry, the computation is carried out on the half-domain
$\Omega=[0,1]\times[0,1]$ with substrate $\Gamma=[0,1]\times\{0\}$ and a
natural symmetry condition on $\{0\}\times[0,1]$.  The initial liquid region is
the rectangle $[0,0.5]\times[0,0.25]$, corresponding to the right half of the
full rectangular droplet $[-0.5,0.5]\times[0,0.25]$.  The spatial
discretization uses a uniform $128\times128$ rectangular mesh, so that
$h=1/128$, with polynomial degree $k=1$; the diffuse-interface thickness is
$\delta=0.01$.  Both the Allen--Cahn and Cahn--Hilliard computations use the
target time step $\Delta t=0.01$.  The initial state is deliberately far from
equilibrium: its corners generate bulk curvature relaxation, while the contact
line adjusts to the Young angle imposed by the wall energy.

For each model, we monitor the discrete free-energy decay rate
\[
  r_E^n=\frac{E(\phi_h^{n-1})-E(\phi_h^n)}{\Delta t}
\]
and stop at the first step for which $r_E^n<10^{-7}$.  The first three
interface panels compare the two models at the common physical times
$t=0.02$, $0.2$, and $1.0$.
The last panel compares the terminal
states after the stopping criterion has been reached.  For
$\theta_Y=45^\circ$, the terminal times are $t=3.44$ for Allen--Cahn and
$t=3.30$ for Cahn--Hilliard; for $\theta_Y=135^\circ$, they are $t=2.39$ and
$t=2.22$, respectively.

Figures~\ref{fig:2d_theta45_interfaces} and
\ref{fig:2d_theta135_interfaces} compare the zero level sets of the two
phase-field solutions with the exact sharp-interface circular cap of the same
area and Young angle.  The acute case spreads along the substrate, while the
obtuse case retracts from the initial rectangular footprint.  In both cases,
the late-time interfaces are close to the circular-cap equilibrium.  The
Allen--Cahn and Cahn--Hilliard transients are not identical, but their
difference decreases as the common equilibrium is approached.

\begin{figure}[H]
\centering
\includegraphics[width=0.96\textwidth]{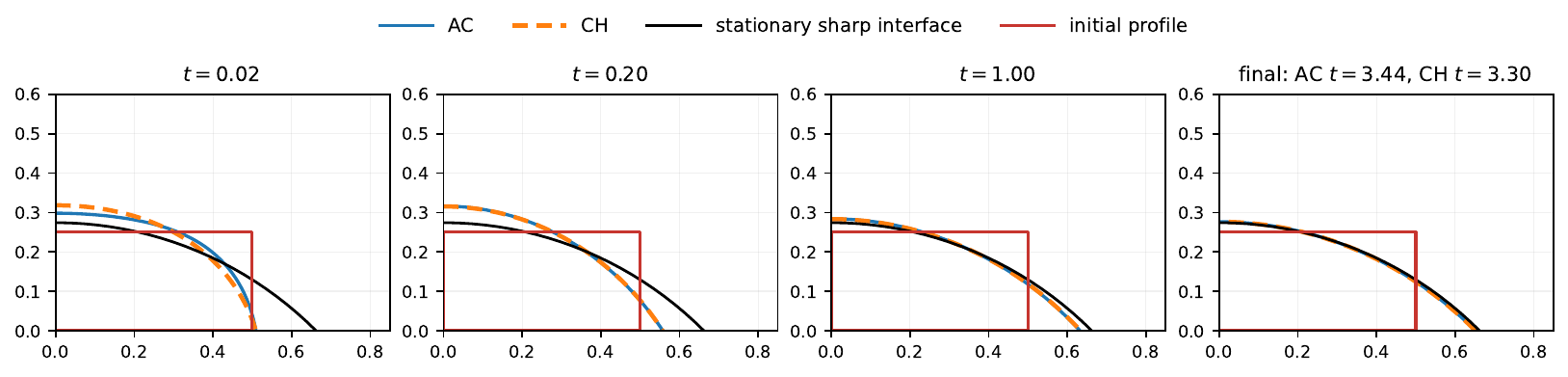}
\caption{Two-dimensional wetting for $\theta_Y=45^\circ$.  The solid blue
curve is the Allen--Cahn zero level set, the dashed orange curve is the
Cahn--Hilliard zero level set, the solid black curve is the stationary
sharp-interface circular cap, and the solid red curve is the initial profile.  The last panel
compares the terminal states at $t=3.44$ for Allen--Cahn and $t=3.30$ for
Cahn--Hilliard.}
\label{fig:2d_theta45_interfaces}
\end{figure}

\begin{figure}[H]
\centering
\includegraphics[width=0.96\textwidth]{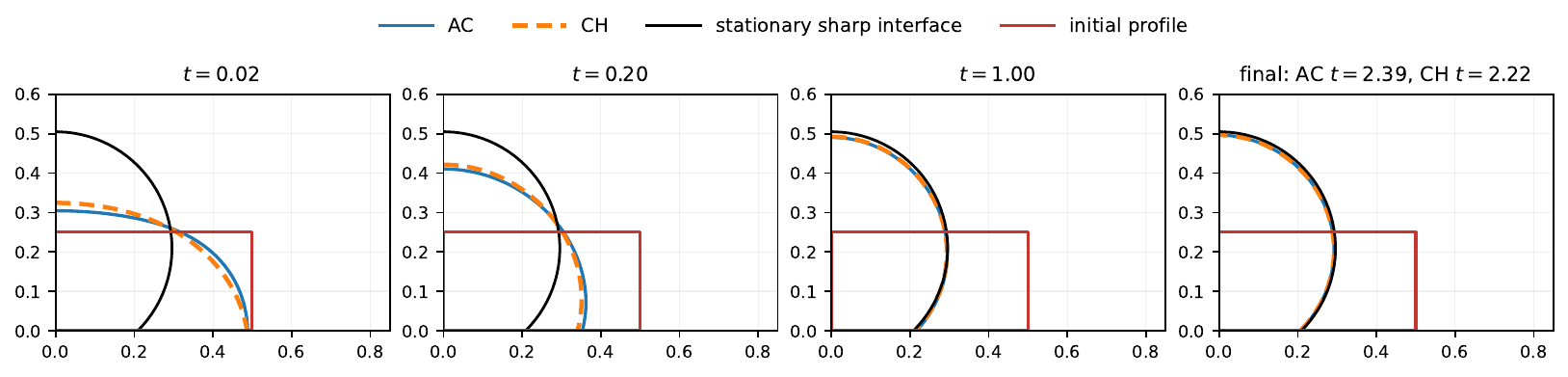}
\caption{Two-dimensional dewetting for $\theta_Y=135^\circ$.  The line
conventions are the same as in Figure~\ref{fig:2d_theta45_interfaces}.  The
droplet retracts from the initial rectangular footprint and approaches the
obtuse circular-cap equilibrium.  The last panel compares the terminal states
at $t=2.39$ for Allen--Cahn and $t=2.22$ for Cahn--Hilliard.}
\label{fig:2d_theta135_interfaces}
\end{figure}

The corresponding energy histories are shown in Figure~\ref{fig:2d_energy}.
The first row gives the wetting case, and the second row gives the dewetting
case; within each row, the left panel shows the total free energy and the right
panel shows the decay rate on a logarithmic scale.  In both contact-angle
regimes, the energy decreases rapidly during the initial corner smoothing and
contact-line adjustment, and then levels off as the cap becomes nearly
stationary.  The dotted line marks the threshold $10^{-7}$ used to choose the
terminal states.  The Cahn--Hilliard curve follows the same energetic trend,
with small differences in the transient decay rate caused by the $H^{-1}$ bulk
metric.

\begin{figure}[H]
\centering
\includegraphics[width=0.72\textwidth]{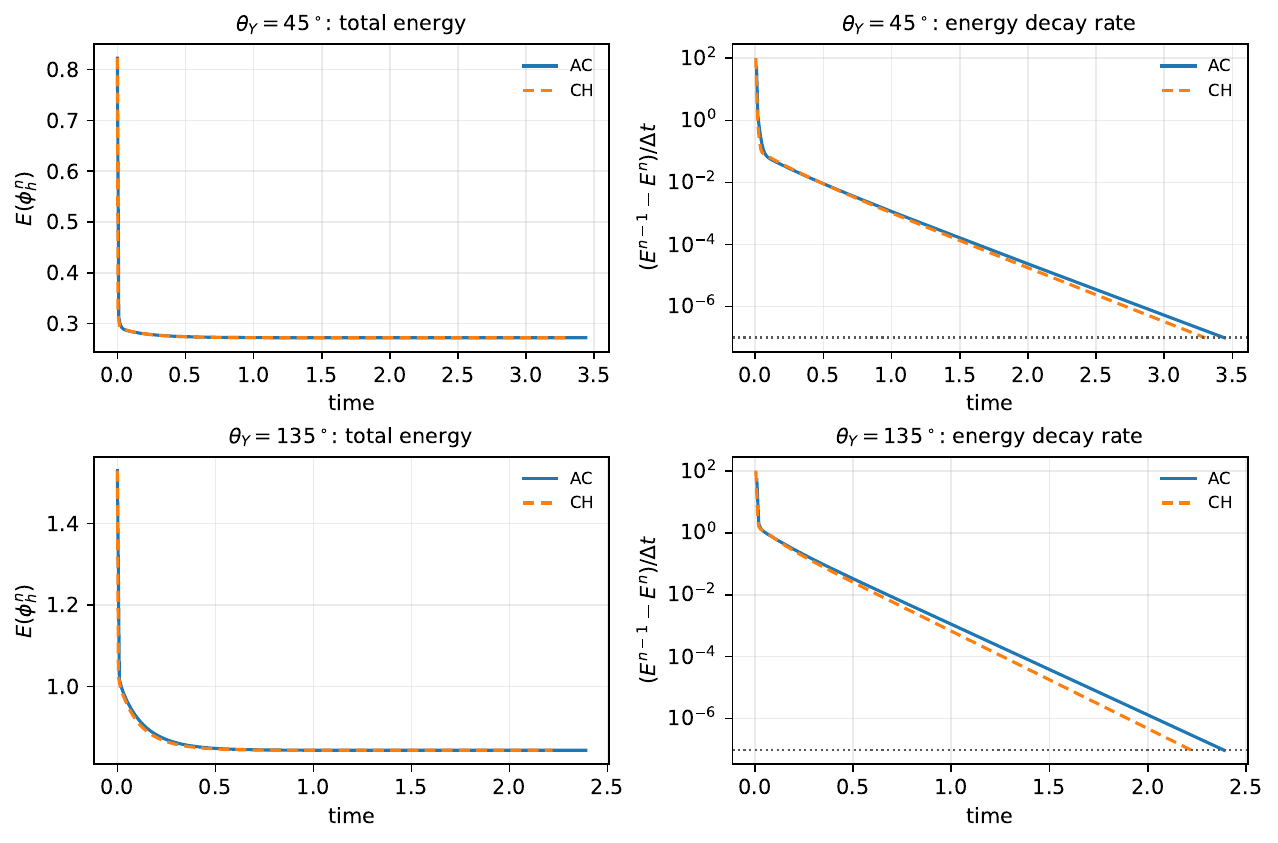}
\caption{Discrete free-energy histories for the two-dimensional wetting and
dewetting tests.  The first row corresponds to $\theta_Y=45^\circ$, and
the second row corresponds to $\theta_Y=135^\circ$.  The left column shows
$E(\phi_h^n)$, and the right column shows
$(E(\phi_h^{n-1})-E(\phi_h^n))/\Delta t$ on a logarithmic scale.  The dotted
line indicates the threshold $10^{-7}$ used to select the terminal states.}
\label{fig:2d_energy}
\end{figure}

Mass conservation is shown in Figure~\ref{fig:2d_mass}.  We plot the drift
$(\phi_h^n-\phi_h^0,1)_\Omega$ rather than the absolute mass, since the latter
is visually indistinguishable from a constant on the scale of the computation.
The drift remains at the level of roundoff and nonlinear solver tolerance in
all four runs.  For Allen--Cahn, conservation follows from the scalar constraint
\eqref{eq:AC_FEM_b}; for Cahn--Hilliard, it follows from
\eqref{eq:CH_FEM_h_a} by testing with the constant function.

\begin{figure}[H]
\centering
\includegraphics[width=0.72\textwidth]{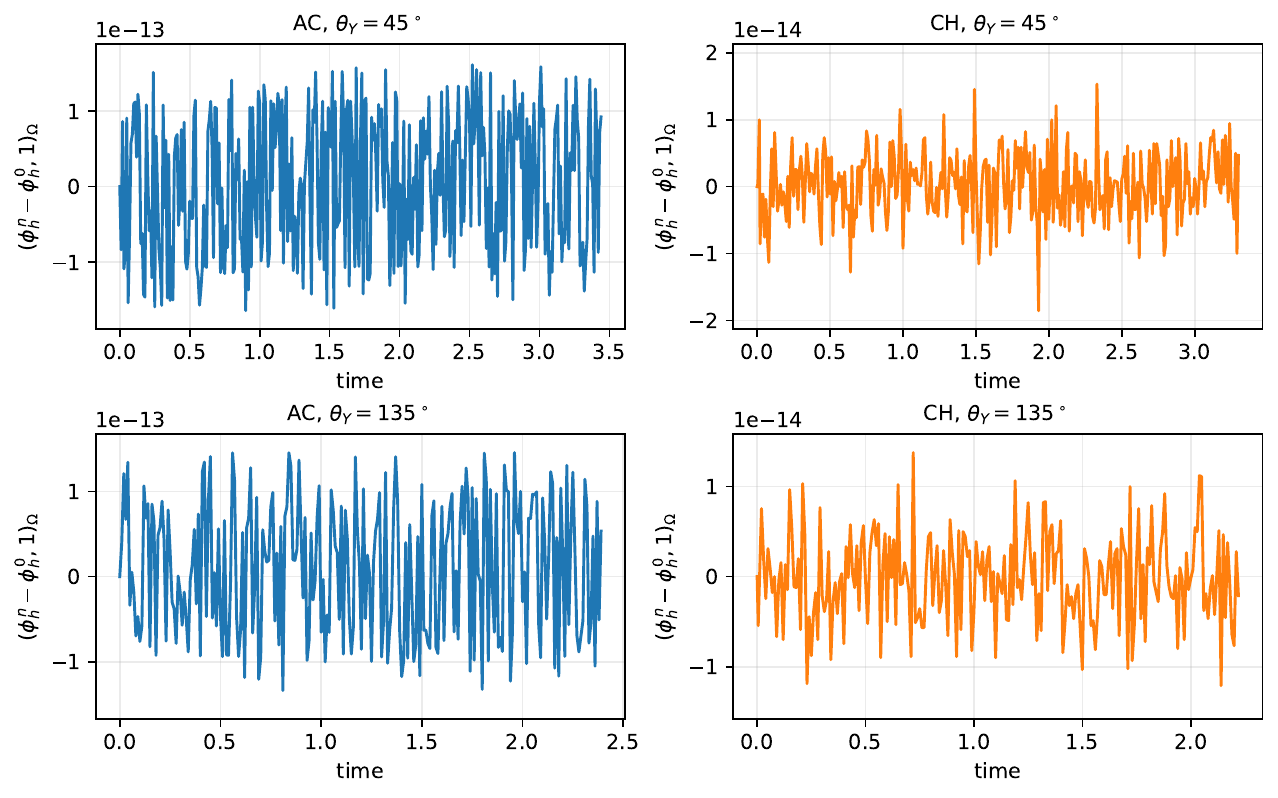}
\caption{Discrete mass conservation for the four two-dimensional runs.  Each
panel shows the drift of the conserved quantity
$(\phi_h^n,1)_\Omega$ from its initial value.}
\label{fig:2d_mass}
\end{figure}

\subsubsection{Three-dimensional wetting and dewetting.}
We next use the same setup in three dimensions.  Because the equilibrium
droplet is axisymmetric, the computation is carried out on the quarter domain
$\Omega=[0,1]^3$, with symmetry conditions on the planes $x=0$ and $y=0$ and
the substrate on $z=0$.  The initial liquid region is the
rectangular block $[0,0.5]\times[0,0.5]\times[0,0.25]$, corresponding to a full
liquid volume $0.25$ after reflection across the symmetry planes.  Equivalently,
the conserved phase-field mass on the quarter domain is
$(\phi_h^0,1)_\Omega=0.875$.  The spatial discretization uses a uniform
$128\times128\times128$ mesh with $k=1$, the diffuse-interface thickness is
$\delta=0.01$, and the target time step is $\Delta t=0.01$.  We stop the
simulation when the discrete energy decay rate falls below $10^{-7}$.

For a homogeneous substrate, the sharp-interface equilibrium in three
dimensions is a spherical cap with the prescribed Young angle and the same
liquid volume.  We therefore compare the numerical zero level set with this
exact cap.  Figure~\ref{fig:3d_theta45_xz} shows the $y=0$ cross-section for
$\theta_Y=45^\circ$.  As in the two-dimensional wetting case, the initially
rectangular droplet spreads along the substrate and relaxes toward the
spherical-cap profile.  The final time is $t=2.69$, at which point the
energy-decay stopping criterion is reached.

\begin{figure}[H]
\centering
\includegraphics[width=0.96\textwidth]{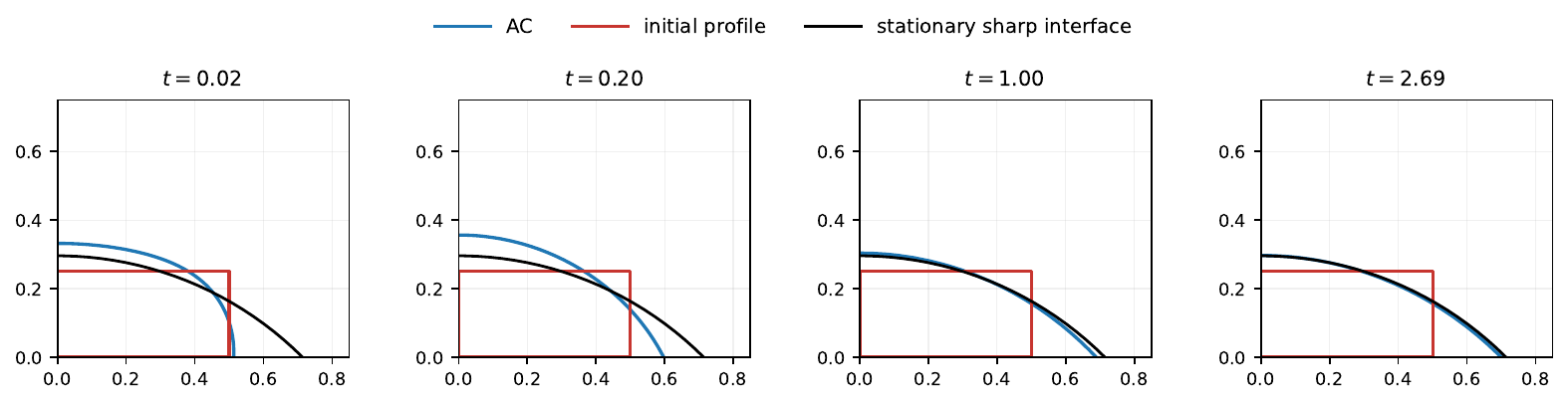}
\caption{Three-dimensional Allen--Cahn wetting for $\theta_Y=45^\circ$.
The figure shows the $y=0$ cross-section on the simulated quarter domain.  The
solid blue curve is the numerical zero level set, the solid red curve is the
initial profile, and the solid black curve is the stationary sharp-interface
spherical cap with the same liquid volume and Young angle.}
\label{fig:3d_theta45_xz}
\end{figure}

Figure~\ref{fig:3d_theta135_xz} shows the corresponding dewetting computation
for $\theta_Y=135^\circ$.  The same initial block is used, but the wall energy
now favors a much smaller wetted area.  The droplet retracts from the initial
footprint and develops the overhanging profile expected for an obtuse
Young angle.  The snapshots show the same qualitative behavior as in
the two-dimensional dewetting test.  The final time is $t=2.80$, at which
point the energy-decay stopping criterion is reached.
In both three-dimensional cases, the numerical zero level set is visibly
slightly inside the sharp-interface curve.  This is consistent with the finite
diffuse-interface thickness $\delta=0.01$ used in the computation.

\begin{figure}[H]
\centering
\includegraphics[width=0.96\textwidth]{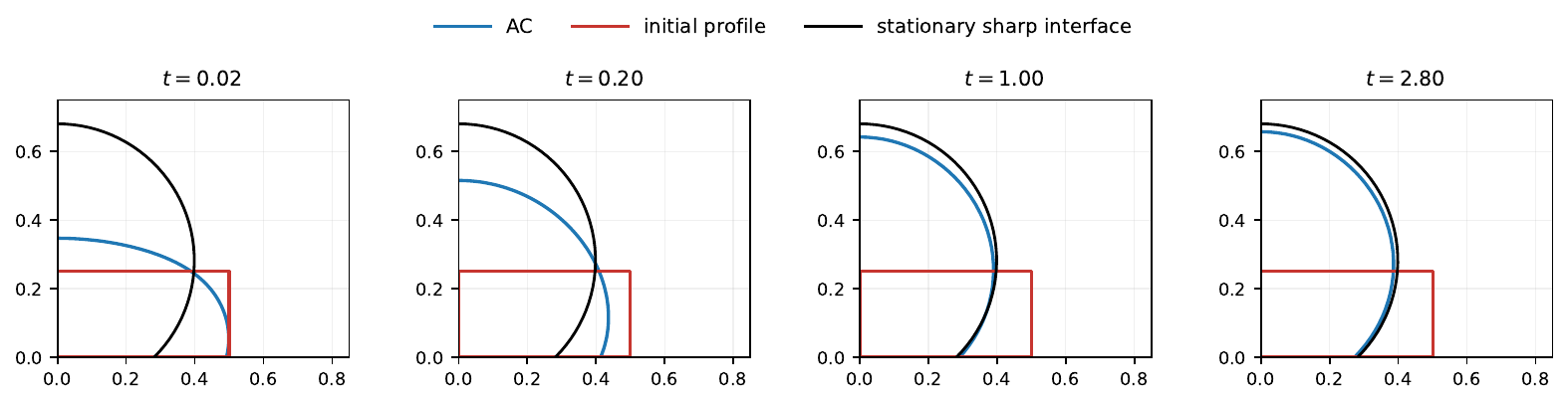}
\caption{Three-dimensional Allen--Cahn dewetting for $\theta_Y=135^\circ$.
The figure shows the $y=0$ cross-section on the simulated quarter domain.  The
line conventions are the same as in Figure~\ref{fig:3d_theta45_xz}.}
\label{fig:3d_theta135_xz}
\end{figure}

The energy histories are shown in Figure~\ref{fig:3d_energy}.  For both contact
angles, the free energy decreases rapidly during the initial relaxation and then
settles into a slow decay.  Both runs reach the threshold $10^{-7}$.  The trend
is consistent with the two-dimensional tests: the variational structure drives
monotone relaxation toward the cap selected by Young's law.

\begin{figure}[H]
\centering
\includegraphics[width=0.60\textwidth]{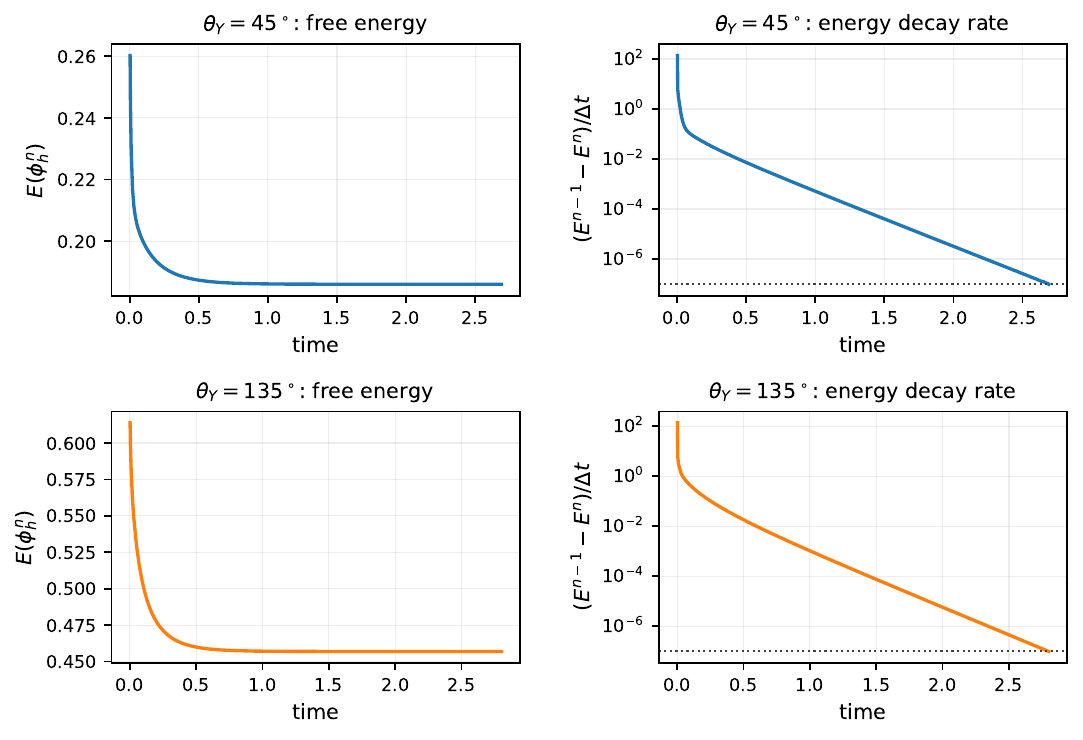}
\caption{Discrete free-energy histories for the three-dimensional Allen--Cahn
wetting and dewetting tests.  The first row corresponds to
$\theta_Y=45^\circ$, and the second row corresponds to
$\theta_Y=135^\circ$.  The left column shows $E(\phi_h^n)$, and the right
column shows $(E(\phi_h^{n-1})-E(\phi_h^n))/\Delta t$ on a logarithmic scale.
The dotted line indicates the stopping threshold $10^{-7}$.}
\label{fig:3d_energy}
\end{figure}

The same comparison can also be displayed by reflecting the computed interface
across the two symmetry planes.  Figure~\ref{fig:3d_surface_compare} plots the
reconstructed zero level sets together with the exact spherical caps.  The left
panel uses the terminal wetting state for $\theta_Y=45^\circ$, and the right
panel uses the terminal dewetting state for $\theta_Y=135^\circ$.  These
surfaces provide a direct three-dimensional analogue of the two-dimensional
sharp-interface comparisons.

\begin{figure}[H]
\centering
\includegraphics[width=0.68\textwidth]{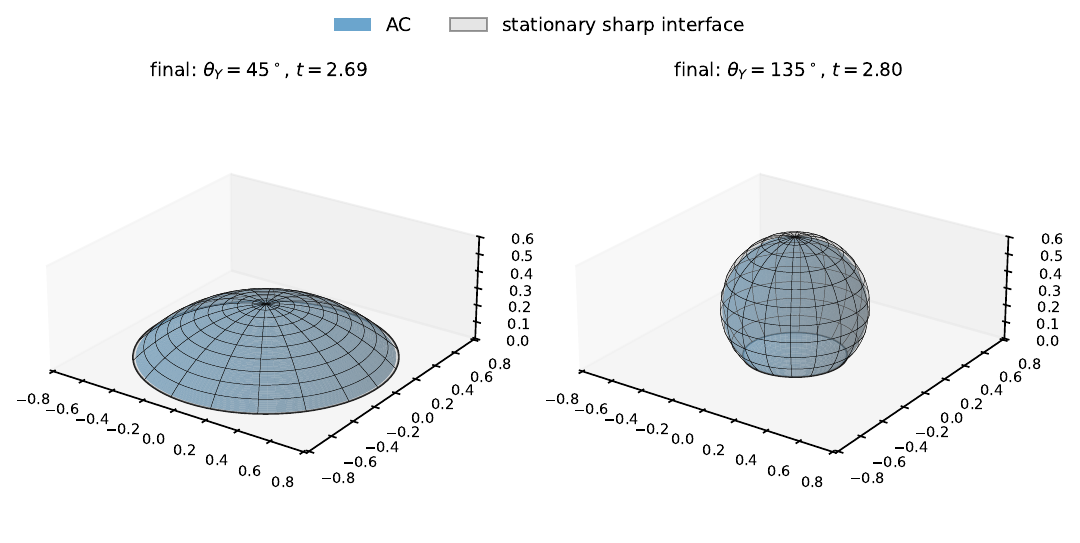}
\caption{Three-dimensional reconstructed interfaces.  The blue surfaces are the
Allen--Cahn zero level sets after reflection across the symmetry planes, and
the gray surfaces are the corresponding sharp-interface spherical caps with the
same liquid volume.  The left panel shows the terminal wetting state
$(\theta_Y=45^\circ)$; the right panel shows the terminal dewetting state
$(\theta_Y=135^\circ)$.}
\label{fig:3d_surface_compare}
\end{figure}

\section{Conclusion}

In this paper, we studied Allen--Cahn and Cahn--Hilliard phase-field models for
contact line dynamics and their sharp interface limit from a unified Onsager's  variational point of view.  Starting from
the same bulk and wall free energy, both models encode the relaxation of the
capillary interface and the motion of the contact line through an
energy-dissipation structure.  The two models differ in the bulk metric:
Allen--Cahn uses a volume-constrained $L^2$ dissipation metric, while Cahn--Hilliard uses an
$H^{-1}$ metric.  This difference changes the transient evolution of the  sharp interface, but not the
  stationary profile selected by the free energy.

The matched asymptotic analysis shows that both phase field models
recover sharp interface contact line dynamics as $\delta\to0$.  For the
Allen--Cahn model, the limiting interface evolves by mean curvature together
with the contact line law driven by the deviation from Young's angle.  For the
Cahn--Hilliard model, the limiting bulk problem is of Mullins--Sekerka type,
with the same contact line relaxation law on the substrate.  The geometric
relation between the interface normal velocity and the contact line velocity is
essential in matching the diffuse and sharp formulations.

We also derived minimizing-movement finite element schemes for both models.
The numerical tests confirm the main qualitative picture: the Allen--Cahn and
Cahn--Hilliard schemes relax toward the same stationary sharp interface spherical cap solution, while their paths to equilibrium reflect the different bulk
dissipation.  In the   wetting and dewetting tests, both models converges to the stationary profile with monotone energy decay.   

\section*{Acknowledgments}
G. Fu was partially supported by NSF DMS-2410740.
Y. Gao was partially supported by NSF DMS-2204288 and  CAREER award DMS-2440651. This work was partially completed while J.-G. Liu was in residence at the Simons Laufer Mathematical Sciences Institute in Berkeley, California, during the Fall 2025 semester.

\section*{Declarations}

The authors have no conflict of interest.

\section*{Data Availibility}

 All data generated or analysed during this study are included in this article and  are available in the data repository at \href{https://github.com/gridfunction/mfem-contact-line-dynamics}{https://github.com/gridfunction/mfem-contact-line-dynamics}.

\end{document}